\documentclass[11pt,reqno]{amsart}
\usepackage{latexsym,amsmath,amssymb, amsthm, mathscinet}
\usepackage{cases, verbatim}
\usepackage{hyperref}

\usepackage{amsfonts,amsmath,amsthm, amssymb}
\usepackage{cases}
\usepackage[usenames]{color}
\usepackage{enumerate} 
\theoremstyle{plain}
\newtheorem{thm}{Theorem}[section]

\newtheorem{defn}{Definition}

\newtheorem{lem}[thm]{Lemma}

\newtheorem{prop}[thm]{Proposition}
\theoremstyle{remark}
\newtheorem{ex}{\bf{Example}}[section]
\newtheorem{rem}[thm]{\bf{Remark}}
\numberwithin{equation}{section}

\newcommand{\R}{\mathbb{R}}
\newcommand{\bS}{\mathbb{S}}

\newcommand{\cB}{\mathcal{B}}

\newcommand{\USC}{{USC\,}}
\newcommand{\LSC}{{LSC\,}}
\newcommand{\UC}{{UC\,}}

\newcommand{\ep}{\varepsilon}

\newcommand{\Gam}{\Gamma}

\newcommand{\Om}{\Omega}

\newcommand{\ol}{\overline}

\newcommand{\Div}{{\rm div}\,}

\newcommand{\tr}{{\rm tr}}

\DeclareMathOperator*{\limsups}{limsup^\ast}
\DeclareMathOperator*{\liminfs}{liminf_\ast}

\begin{document}
\title{Quasiconvexity preserving property for fully nonlinear nonlocal parabolic equations}

\author[T. Kagaya]{Takashi Kagaya}
\address[T. Kagaya]{Graduate school of Engineering, Muroran Institute of Technology, 27-1 Mizumoto-cho, Hokkaido, 050-8585, Japan}
\email{kagaya@mmm.muroran-it.ac.jp}

\author[Q. Liu]{Qing Liu}
\address[Q. Liu]{Geometric Partial Differential Equations Unit, Okinawa Institute of Science and Technology Graduate University, 1919-1 Tancha, Onna-son, Kunigami-gun, 
Okinawa, 904-0495, Japan, 
}
\email{qing.liu@oist.jp}

\author[H. Mitake]{Hiroyoshi Mitake}
\address[H. Mitake]{
Graduate School of Mathematical Sciences, 
University of Tokyo, 
3-8-1 Komaba, Meguro-ku, Tokyo, 153-8914, Japan}
\email{mitake@ms.u-tokyo.ac.jp}

%\thanks{ 
     
%}
%\end{comment}
 %viscosity solutions to PDE
\date{\today}

\maketitle
%\tableofcontents

%%%%%%%%%%%%%%%%%%%%%%%%%%%%%%%%%%%%%%%%%%%%%%%%%%%%%%%%%%%%%%%%%%%%%%%%%%%%%%%%%%%%%%%%%%%%%%%%%%%%%%%%%%%%%%%%%%%%%%%%%%%%%%%%%%%%

\begin{abstract}
This paper is concerned with a general class of fully nonlinear parabolic equations with monotone nonlocal terms. We investigate the quasiconvexity preserving property of positive,  spatially coercive viscosity solutions. 
We prove that if the initial value is quasiconvex, the viscosity solution to the Cauchy problem stays quasiconvex in space for all time. Our proof can be regarded as a limit version of that for power convexity preservation as the exponent tends to infinity. We also present several concrete examples to show applications of our result. 
\end{abstract}

\section{Introduction}

In this paper, we study a class of fully nonlinear nonlocal parabolic equations:
\begin{numcases}{}
u_t+F(u, \nabla u, \nabla^2 u,K\cap \{u(\cdot, t)< u(x,t)\})=0 & in $ \mathbb{R}^n\times(0, \infty)$, \label{E1}\\
u(\cdot,0)=u_0 & in $ \mathbb{R}^n$,  \label{initial}
\end{numcases}
where $u:\R^n\times[0,\infty)\to\R$ is a unknown function,  and $u_t$, $\nabla u$ and $\nabla^2u$ 
denote the time derivative, the spatial gradient and Hessian of $u$, respectively. 
Here $K\subset\R^n$ is a compact set, the initial condition $u_0:\R^n\to\R$ is in $\UC(\R^n)$, 
where $\UC(\R^n)$ is the set of uniformly continuous functions on $\R^n$, 
and $F: \R\times (\R^n\setminus \{0\})\times \bS^{n}\times \cB_K\to \R$ is a given continuous function, where $\bS^n$ denotes the space of $n \times n$ 
real symmetric matrices and $\cB_K$ represents the collection of all measurable subsets of $K$. 

The goal of this paper is to investigate the preservation of spatial quasiconvexity of viscosity solutions to \eqref{E1}, \eqref{initial}. Here, a function $u\in C(\R^n\times [0, \infty))$ is said to be \textit{spatially quasiconvex} if all sublevel sets of $u(\cdot, t)$ are convex in $\R^n$, or equivalently, 
\[
u(\lambda y+(1-\lambda)z, t)\leq \max\{u(y, t), u(z, t)\}
\]
holds for all $y, z\in \R^n$, $t\geq 0$ and $\lambda\in (0, 1)$.

Our work is closely related to \cite{C}, where a general class of set evolutions with nonlocal terms is shown to preserve the convexity of the initial set. A typical example is the level set equation of the surface evolution equation
\begin{equation}\label{eq:ex1}
V=
%\bigl(
a +bm(K\cap\Omega_t)-c\Div_{\Gam_{t}}\xi(n(x)),
%\bigr), 
\end{equation}
%\begin{equation}\label{eq:ex1}
%V=
%%\bigl(
%a M(n(x))-bm(K\cap\Omega_t)-{\color{red}c}\Div_{\Gam_{t}}\xi(n(x)),
%%\bigr), 
%\end{equation}
where $a\in \R$, $b\ge 0$ and $c\ge0$ are given constants,  $V$ is the outward 
\textit{normal velocity} of an evolving compact hypersurface $\Gam_{t}$, $\Omega_t$ is the set enclosed by $\Gamma_t$, $\xi$ denotes the so-called \textit{Cahn-Hoffman vector}, and $n(x)$ denotes the \textit{unit normal} to $\Gam_{t}$ at $x$
pointing outward $\Om_{t}$. 
Here, $m(A)$ represents the $n$-dimensional Lebesgue measure of $A\subset\R^n$ and $\Div_{\Gam_{t}}\xi(n(x))$ stands for the \textit{anisotropic mean curvature} of $\Gam_t$ at $x$. We shall give precise assumptions on $\xi$ in Section \ref{sec:nonlocal-geo}. 

Such nonlocal evolutions can be reformulated via the so-called level set method as geometric parabolic equations, which in our context requires $F$ to satisfy
\begin{equation}\label{geometrical-pro}
\begin{aligned}
F(r_1, c_1 p, c_1X & +c_2 p\otimes p, A)=c_1F(r_2, p, X, A)\\
&\text{for all $c_1, c_2\in \R$, $r_1, r_2\in \R$, $p\in \R^n\setminus \{0\}$, $X\in \bS^{n}$, $A\in\cB_K$};
\end{aligned}
\end{equation}
in particular, $F$ is independent of the unknown. We refer to \cite{GBook} for more details about the level set formulation. Such equations also appear in the singular limit of some nonlocal reaction diffusion equations \cite{CHL}.

It is natural to ask whether it is possible to obtain the same result for nonlocal evolution  equations without  %under weaker assumptions than 
assuming \eqref{geometrical-pro} in order to allow broader applications. In this paper we give an affirmative answer to this question. Under a set of relaxed assumptions, we show the convexity preserving property for all sublevel sets of solutions and present several concrete examples of applications. 

%Throughout this paper, we shall consider only positive viscosity solutions to \eqref{E1}, \eqref{initial}. One can apply our results to more general solutions that are bounded from below by adding appropriate constants to them and slightly changing the equations. 
Let us begin with our assumptions on the operator $F$. %is \textit{degenerate elliptic} and \textit{geometric}, that is,
\begin{enumerate}
\item[(F1)] $F$ is \textit{proper} and \textit{degenerate elliptic}; namely, for any $p\in \R^n\setminus \{0\}$ and $A\in \cB_K$,
\[
F(r_1, p, X, A)\leq F(r_2, p, X, A)
\]
holds for all $r_2\geq r_1$ and $X\in \bS^n$, and 
\begin{equation}\label{elliptic}
F(r, p, X_1, A)\leq F(r, p, X_2, A)
\end{equation}
holds for all $r\in \R$, $X_1, X_2\in\bS^n$ satisfying $X_1\geq X_2$. 
\item[(F2)] $F$ is locally bounded in the sense that for each $R>0$, there holds
\[
\sup\{|F(r, p, X, A)|:  r\in \R, |p|\leq R\text{ with }p\neq 0, |X|\leq R, A\in \cB_K\}<\infty.  
\]
%\ma{where $\tilde{\cB}$ is the space of all measurable sets in $\R^n$}. 
%\item[(F2)] $F$ is locally bounded in the sense that for each $R>0$, there holds
%\[
%\sup\{|F(r, p, X, K\cap A)|:  r\in \R, |p|\leq R\text{ with }p\neq 0, |X|\leq R, A\in \ma{\tilde{\cB}}\}<\infty, 
%\]
%\ma{where $\tilde{\cB}$ is the space of all measurable sets in $\R^n$}. 
\item [(F3)] $F$ is continuous in $\R\times (\R^n\setminus \{0\})\times \bS^{n}\times \cB_K$ with the topology of $\cB_K$ given by $d(A_1, A_2)=m(A_1\triangle A_2)$, where $m$ denotes the Lebesgue measure and $A_1\triangle A_2$ stands for the symmetric difference of $A_1$ and $A_2$, 
that is  $A_1\triangle A_2:=(A_1\cup A_2)\setminus (A_1\cap A_2)$ for all $A_1, A_2\in\cB_K$. 
Moreover, for any $R>0$, there exists a modulus of continuity $\omega_R$ such that
%\begin{equation}\label{uc-hessian}
%|F(p, X, A)-F(p, Y, A)|\leq \omega_R(|X-Y|)
%\end{equation}
%for all $p\in \R^n\setminus \{0\}$ with $|p|\leq R$, $X, Y\in \bS^n$ and $A\in \cB$, and 
\begin{equation}\label{usc-set}
F(r, p, X, A_1)-F(r, p, X, A_2)\leq \omega_R\left(m(A_1\triangle A_2)\right)
\end{equation}
for all $r\in \R$, $p\in \R^n\setminus \{0\}$ with $|p|\leq R$, $X\in \bS^n$ and $A_1, A_2\in \cB_K$. 
\item[(F4)] $F$ is \textit{monotone} with respect to the set argument; namely, 
\[
F(r, p, X, A_1)\leq F(r, p, X, A_2)
\] 
holds for all $r\in \R$, $p\in \R^n\setminus \{0\}$, $X\in \bS^n$, $A_1, A_2\in \cB_K$ with $A_1\subset A_2$. 
\item[(F5)] $F$ satisfies the \textit{structure condition} for the comparison principle: 
There exists a modulus of continuity $\omega:[0,\infty)\to[0,\infty)$ %satisfying $\omega(0)=0$ 
such that 
\[
F(r, p_1, X_1, A)-F(r, p_2, X_2, A)\le \omega\left(\frac{|Z||p_1-p_2|}{\min\{|p_1|, |p_2|\}}+|p_1-p_2|+\alpha\right) 
\]
for all $\alpha \ge 0$, $r\in \R$, $p_1, p_2\in \R^n\setminus\{0\}$, $A\in \cB_K$, and $X_1, X_2\in\bS^n$ satisfying 
\[
\begin{pmatrix}
X_1 & 0\\ 
0 & -X_2
\end{pmatrix} 
\le 
\begin{pmatrix}
Z & -Z\\ 
-Z & Z
\end{pmatrix} 
+
\alpha
\begin{pmatrix}
I & I\\ 
I & I
\end{pmatrix} 
\]
for $Z\in\bS^n$, where we denote by $I$ the $n\times n$ identity matrix. 
\end{enumerate} 

\begin{enumerate}
\item[(F6)] There exists $\mu\in C(\R)$ such that
\[
 \sup_{r\in \R,\ A\in \cB_K} |F(r, p, X, A)-\mu(r)|\to 0\quad \text{as $(p, X)\to (0, 0)$.}
 \]
 \end{enumerate}
The assumption (F6) immediately implies that $F^\ast(r, 0, 0, A)=F_\ast(r, 0, 0, A)=\mu(r)$
for all $r>0$ and $A\in \cB_K$, where we denote by $F_\ast, F^\ast$ the upper and lower semicontinuous envelope of $F$, respectively (see \cite{CIL, GBook} for definitions). 

The ellipticity \eqref{elliptic} of $F$ yields 
\[
F_\ast(r, 0, X_1, A)\leq F_\ast(r, 0, X_2, A), \quad F^\ast(r, 0, X_1, A)\leq F^\ast(r, 0, X_2, A),
\]
for all $r\in \R$, $X_1, X_2\in\bS^n$ satisfying $X_1\geq X_2$. In particular, we have
\[
F^\ast(r, 0, X_1, A)\leq \mu(r)\leq F_\ast(r, 0, X_2, B)\quad \text{for all $r\in \R$, $X_1\geq 0\geq  X_2$ and $A, B\in \cB_K$.}
\]

We stress that the monotonicity (F4) is crucial for validity of the comparison principle, which facilitates our analysis. See \cite{C, Sl, DKS, Sr} for comparison results for monotone evolution equations in different settings. On the other hand, in the non-monotone case, one cannot expect the comparison principle to hold and alternative methods are needed to prove uniqueness of solutions and other related properties (see \cite{ACM, BL, BLM, KK} for instance).

%The coercivity condition is guaranteed by the comparison principle and the existence of a specific coercive subsolution, which will be explained in a moment. 
% One can apply our results to more general solutions that are bounded from below by adding appropriate constants to them and slightly changing the equations. 
In order to clearly demonstrate our arguments for the convexity property,  throughout this paper we shall only consider viscosity solutions to \eqref{E1}, \eqref{initial} that are uniformly positive and coercive in space, that is,  
\begin{equation}\label{ineq:positive}
u\ge c_0\quad\text{in} \  \R^n\times [0, \infty), \quad \text{for some $c_0>0$}
\end{equation} 
and
\begin{equation}\label{coercive}
\inf_{|x|\geq R,\ t \le T}u(x, t)\to \infty\quad \text{as $R\to \infty$ for any $T\geq 0$.}
\end{equation}
We actually include further assumptions on $u_0\in \UC(\R^n)$ to guarantee the existence and uniqueness of solutions $u\in C(\R^n\times [0, \infty))$ of \eqref{E1}, \eqref{initial}, satisfying \eqref{ineq:positive}, \eqref{coercive}. 
Moreover, we can show that
\begin{equation}\label{uc-preserving}
|u(x, t)-u(y, t)|\leq \omega_0(|x-y|) \quad \text{for all $x, y\in \R^n$ and $t\geq 0$,}
\end{equation}
where $\omega_0$ denotes the modulus of continuity of $u_0$. See Theorem \ref{thm:exist} for details. 

In addition, we impose a key concavity condition on $F$ via a transformed operator $G_\beta$ defined by 
\begin{equation}\label{transformed op}
\begin{aligned}
&G_\beta(r, p, X, A)\\
&={1\over 1-\beta} r^\beta
 F\left(r^{1-\beta}, (1-\beta) r^{-\beta}p, (1-\beta) r^{-\beta} X+(\beta^2-\beta)r^{-\beta-1} p\otimes p, A\right) 
\end{aligned}
\end{equation}
for $0<\beta<1$, $r>0$, $p\in \R^n\setminus \{0\}$, $X\in\bS^n$ and $A\in \cB_K$.
\begin{enumerate}
\item[{\rm(F7)}] 
For any $\beta<1$ close to $1$,  
\begin{equation}\label{concave1}
(r, X)\mapsto G_\beta(r, p, X, A)\quad \text{is concave in $[c_0, \infty)\times \bS^n$}
\end{equation}
for any $p\in \R^n\setminus\{0\}$ and $A\in \cB_K$, and 
\begin{equation}\label{concave2}
r\mapsto r^\beta \mu(r^{1-\beta})\quad \text{is concave in $[c_0, \infty)$, 
}
\end{equation}
where $\mu$ is given by (F6).

%and for all $r_1, r_2>0$,  and $\lambda\in[0,1]$
%\begin{equation}
%\begin{aligned}
%\lambda G_\beta(r_1, p, X_1, A)&+(1-\lambda)G_\beta(r_2, p, X_2, A)\\
%&\leq G_\beta(\lambda r_1+(1-\lambda)r_2, p, \lambda X_1+(1-\lambda)X_2, A)
%\end{aligned}
%\end{equation}
%\begin{equation}
%\lambda F(p, X, A)+(1-\lambda)F(p, Y, A)\leq F(p, \lambda X+(1-\lambda)Y, A)
%\end{equation}
\end{enumerate}

%under the 
%the quasiconvexity assumption on $u_0$. 

Let us state our main result. 
\begin{thm}[Quasiconvexity preserving property]\label{thm:main}
Assume {\rm(F1)--(F7)}. 
Let $u_0\in \UC(\R^n)$. Let $u\in C(\R^n\times [0, \infty))$ be the unique viscosity solution of \eqref{E1}, \eqref{initial} satisfying \eqref{ineq:positive}, \eqref{coercive} and \eqref{uc-preserving}. %\bl{Assume in addition that $u(\cdot, t)$ is uniformly continuous for all $t\geq 0$.} 
If $u_0$ is quasiconvex in $\R^n$, that is, 
$\{u_0<h\}$ is convex for all $h\in\R$, 
 then $u(\cdot, t)$ is quasiconvex in $\R^n$ for all $t\geq 0$.  
\end{thm}
Our result is applicable to a general class of nonlinear parabolic equations including the level set equations for nonlocal geometric evolutions studied in \cite{C}. See Section \ref{sec:applications} for applications of Theorem \ref{thm:main} to some concrete equations. We emphasize that these equations do not need to be geometric. The assumptions in Theorem \ref{thm:main} actually allow the operator $F$ to additionally depend on the unknown $u$. It is worth pointing out that the preservation of quasiconvexity can be obtained also for spatially Lipschitz solutions of a class of nonlinear local diffusion equations.
In Section \ref{sec:vishj}, as an example, we discuss a viscous Hamilton-Jacobi equation with a certain sublinear gradient term. Our quasiconvexity result suggests an interesting convexification effect of the gradient term, since it is known that the heat equation (without gradient terms) fails to preserve quasiconvexity of solutions in the space variable as shown in \cite{IS, IST2} (see \cite{HNS, CW} for counterexamples in other settings). 
In fact, the Laplace operator does not fulfill (F7). See Section \ref{sec:vishj} for more detailed discussions.

The novelty of our work lies not only at the more relaxed setting and wider applications, but also at an improved method in our proof. Instead of the set-theoretic arguments in  \cite{C}, we develop a PDE-based approach. 
%adapting the well-known arguments for local problems to our nonlocal setting. 
Convexity/concavity of solutions is a classical topic of geometric properties of elliptic and parabolic equations. A non-exhaustive list of references includes \cite{Ko, Ka, Ke, CGM, BP} for classical solutions and \cite{GGIS, ALL, Ju, LSZ} for viscosity solutions. A generalized type, called power convexity/concavity, is investigated in \cite{IS1, IS3, ILS, IST} for various equations. One major idea %to prove convexity or power convexity 
applied in \cite{Ka, ALL, ILS} is to show the corresponding convex envelope of a solution is a supersolution of the equation and then use the comparison principle to conclude the proof. This machinery can be implemented also for quasiconvexity or quesiconcavity of classical solutions to the Dirichlet boundary problems for semilinear elliptic or parabolic equations \cite{DK, CS, IS0}.
For viscosity solutions of our fully nonlinear nonlocal problem, we essentially adopt the same strategy but incorporate an additional limit process, based on the fact that quasiconvexity can be regarded as a limit notion of the power convexity as the exponent tends to infinity. 

 For a fixed $\lambda\in (0, 1)$ and a given positive viscosity solution $u$ of \eqref{E1}, 
 in order to prove that the \textit{spatially quasiconvex envelope} $u_{\star, \lambda}$, defined by 
 \begin{equation}\label{eq:envelop2-lam}
 \begin{aligned}
 u_{\star, \lambda}(x, t)=\inf\bigg\{\max \{u(y, t), u(z, t)\}: x=\lambda y+ & (1-\lambda)z  \bigg\},\\
&\text{for } (x, t)\in \R^n\times (0, \infty),
\end{aligned}
 \end{equation}
 is a viscosity supersolution, we use the \textit{spatially power convex envelope} $u_{q, \lambda}$ defined by \eqref{eq:q-envelop} with exponent $q>1$ to approximate $u_{\star, \lambda}$ as $q\to \infty$. The concavity condition (F7), with the choice $\beta=1-{1/q}$, plays an important role of relating $u_{q, \lambda}$ to the supersolution property of \eqref{E1} for $q>1$ arbitrarily large. Also, thanks to the coercivity of $u$ in \eqref{coercive}, the large exponent approximation by $u_{q, \lambda}$ can be conducted in the sense of locally uniform convergence and thus the supersolution property can be passed on to $u_{\star, \lambda}$. More details will be given in Section~\ref{sec:quasiconvex}.

It is worth pointing out that although the power convex envelope is used as a key ingredient in our proof, in general one cannot expect the preservation of power convexity in the space variable, due to the nonlocal structure of the equation. We will clarify this via  Example \ref{ex:1} in Section~\ref{sec:quasiconvex}.
Since quasiconvexity is regarded as the weakest notion in the category of any power convexity, in this sense our generalized convexity preserving result in Theorem \ref{thm:main} can be considered to be optimal. 

As mentioned above, we need a comparison principle to complete the whole proof. Since the comparison results in the literature concerning nonlocal equations are only for bounded solutions, for our own purpose we include a comparison theorem, Theorem \ref{thm:comp}, for possibly unbounded solutions satisfying growth condition \eqref{growth comp}. Our proof is an adaptation of that in \cite{GGIS} to nonlocal problems. It would be interesting to obtain a similar quasiconvexity result  for non-monotone evolution equations, for which the comparison principle fails to hold in general.

We remark that the presence of the compact set $K$ in \eqref{E1} largely facilitates our arguments in this work. It enables us to show the unique existence of solutions that are coercive in space. The case when $K$ is unbounded seems more challenging. 

%Although the elliptic case is not considered in this work, one can extend our method to a class of nonlocal monotone elliptic problems under analogous assumptions on $F$. We refer to \cite{CS} for related results for semilinear elliptic equations without nonlocal terms in nonconvex domains such as ring-shaped regions. The quasiconvexity preserving property for parabolic equations is considered to be more complicated, as it fails to hold even for the heat equation. It is not clear to us if it is possible to further weaken our assumptions in Theorem \ref{thm:main}.

Our interests in \eqref{E1} are not restricted to the quasiconvexity of solutions.  In the upcoming work \cite{KLM2}, we give an optimal control interpretation for a class of first order nonlocal evolution equations and use it to study the large time behavior of solutions, which is not studied much. 

The rest of the paper is organized in the following way. In Section \ref{sec:pre} we recall the definition and some basic properties of  viscosity solutions of \eqref{E1}. We provide a comparison principle for our later applications in Section \ref{sec:cp}. Section \ref{sec:quasiconvex} is devoted to the proof of our main result, Theorem \ref{thm:main}. In Section \ref{sec:applications} we present several concrete examples with further discussions on the assumptions, especially (F7).

\subsection*{Acknowledgments}
TK was partially supported by Japan Society for the Promotion of Science (JSPS) through grants: KAKENHI \#18H03670, \#19K14572, \#20H01801, \#21H00990. 
QL was partially supported by the JSPS grant KAKENHI \#19K03574.
HM was partially supported by the JSPS grants: KAKENHI \#19K03580, \#20H01816, \#21H04431, \#19H00639.

%dependence of the velocity, and 
%on the set $\Om_{t}$ enclosed by $\Gam_{t}$. 
%Here, 
%$M:\partial B(0,1)\to(0,\infty)$, $\xi:\partial B(0,1)\to\R^n$ are give functions which satisfy suitable assumptions 
%(see Remark \ref{??}{\color{red}後でどこかに加える？}). 
%The term $\Div_{\Gam_{t}}\xi(n(x))
%:=\tr\bigl((I-n(x)\otimes n(x))D_{x}\xi(n(x))\bigr)$ 
%is called the \textit{anisotropic} 
%\textit{mean curvature} of $\Gam_{t}$ at $x$ 
%in the direction of $n(x)$ (see \cite{GBook} for instance). 

\begin{comment}
{\color{red}Give a literature description on nonlocal geometric flows here.
\begin{enumerate}
\item
Monotone: \cite{C, Sl, DKS, Sr}
\item
Non-monotone: \cite{ACM, BL, BLM}. 
\end{enumerate}
}

{\color{red}
Give a literature description on convexity properties on elliptic/parabolic equations. \\
\begin{enumerate}
\item
Convexity property: \cite{GGIS, IS} (石毛先生の論文をもう少し引用したいがどれが最適か？)
\item
We should explain the fact that the quasi-convexity is the best ``convexity" from the view point of ``$q$-convexity" 
by using the control interpretation and explicit solution. 

\item
We should anounce our coming next project on the control interpretation here too.   
\end{enumerate}
}
\end{comment}

\section{Preliminaries}\label{sec:pre} 
In this section, we recall some basic results on \eqref{E1}. 
We first recall the definition of viscosity solutions to \eqref{E1}. 
For a set $Q\subset \R^n\times [0,\infty)$, we denote by $\USC(Q)$ and $\LSC(Q)$, respectively, 
the set of the upper and lower semicontinuous functions in $Q$. 
\begin{defn}[Viscosity solutions]\label{def env-sol}
{\rm(i)} A function $u\in\USC(\R^n\times(0,\infty))$ is called a viscosity subsolution of \eqref{E1} if 
%\ma{$u\in\USC(\mathbb{R}^n\times(0,\infty)$ {\rm(}resp., $u\in\LSC(\mathbb{R}^n\times(0,\infty))${\rm)}}, 
whenever there exist $(x_0, t_0)\in \R^n\times (0, \infty)$ and $\varphi\in C^2(\R^n\times (0, \infty))$ such that $u-\varphi$ attains a local maximum at $(x_0, t_0)$, 
\[
\varphi_t(x_0, t_0)+ F_\ast(u(x_0, t_0), \nabla \varphi(x_0, t_0), \nabla^2 \varphi(x_0, t_0), K\cap\{u(\cdot, t_0)< u(x_0, t_0)\})\leq 0.
\]
%The inequality reduces to 
%\[
%\varphi_t(x_0, t_0)\leq 0
%\]
%provided that $\nabla \varphi(x_0, t_0)=0$, regardless of the value $m(\{u(\cdot, t)<u(x_0, t_0)\})$. 
{\rm(ii)} A function $u\in\LSC(\R^n\times(0,\infty))$ is called a viscosity supersolution of \eqref{E1} if whenever there exist $(x_0, t_0)\in \R^n\times (0, \infty)$ and $\varphi\in C^2(\R^n\times (0, \infty))$ such that $u-\varphi$ attains a minimum at $(x_0, t_0)$, 
\[
\varphi_t(x_0, t_0)+F^\ast(u(x_0, t_0), \nabla \varphi(x_0, t_0), \nabla^2 \varphi(x_0,t_0), K\cap\{u(\cdot, t_0)\leq u(x_0, t_0)\})\geq 0.
\]
%The inequality reduces to 
%\[
%\varphi_t(x_0, t_0)\geq 0
%\]
%provided that $\nabla \varphi(x_0, t_0)=0$, regardless of the value $m(\{u(\cdot, t)\leq u(x_0, t_0)\})$. 
{\rm(iii)} A function $u\in C(\R^n\times (0, \infty))$ is called a viscosity solution of \eqref{E1} if it is both a viscosity subsolution and a viscosity supersolution. 
\end{defn}

We are always concerned with viscosity solutions in this paper, and the term ``viscosity" is omitted henceforth.

%\begin{rem}
%As usual, one may replace the local maximum and local minimum in both definitions above by strict local maximum and strict local minimum. Also, we may alternatively define the subsolutions and supersolutions by adopting the semijets $\overline{P}^{2, +}u(x_0, t_0)$ and $\overline{P}^{2, -}u(x_0, t_0)$ to replace the test functions. See \cite{CIL, GBook} for more details. 
%\end{rem}

%The above two types of definitions are actually equivalent, as in the local case. 

\begin{prop}[Refined test functions]\label{prop:equiv}
Assume {\rm(F1)--(F4)} and {\rm(F6)} hold. Let $u\in LSC(\R^n\times (0, \infty))$ (resp., $u\in USC(\R^n\times (0, \infty)$) be a locally bounded function in $\R^n\times (0, \infty)$. 
Then $u$ is a supersolution {\rm(}resp., subsolution{\rm)} of \eqref{E1} %in the sense of Definition \ref{def env-sol} 
if whenever there exist $(x_0, t_0)\in \R^n\times (0, \infty)$ and $\varphi\in C^2(\R^n\times (0, \infty))$ such that $u-\varphi$ attains a local minimum {\rm(}resp., maximum{\rm)} at $(x_0, t_0)$, the following conditions hold:
\begin{itemize}
\item If $\nabla \varphi(x_0, t_0)\neq 0$, then 
\begin{equation}\label{f-super}
\varphi_t(x_0, t_0)+ F(u(x_0, t_0), \nabla \varphi(x_0, t_0), \nabla^2 \varphi(x_0, t_0), K\cap\{u(\cdot, t_0)\leq u(x_0, t_0)\})\geq 0
\end{equation}
\[%\label{f-sub}
\left(\text{resp., }\ \varphi_t(x_0, t_0)+ F(u(x_0, t_0), \nabla \varphi(x_0, t_0), \nabla^2 \varphi(x_0, t_0), K\cap\{u(\cdot, t_0)< u(x_0, t_0)\})\leq 0
\right). 
\]
\item If $\nabla \varphi(x_0, t_0)= 0$ and $\nabla^2\varphi(x_0, t_0)=0$, then 
\begin{equation}\label{f-singular}
\varphi_t(x_0, t_0)+\mu(u(x_0, t_0)) \geq 0 \quad \left(\text{resp.,} \ \varphi_t(x_0, t_0)+\mu(u(x_0, t_0))\leq 0\right).
\end{equation}
\end{itemize}
\end{prop}

We refer to \cite[Proposition 2.2.8]{GBook} for a similar property in the case of local singular equations.   
In order to prove Proposition \ref{prop:equiv}, let us first give an elementary result (see also \cite[Equation (5)]{Sl}). 
\begin{lem}[Continuity of measures]\label{lem meas}
Let $a\in \R$ and $\{a_\mu\}_{\mu>0}\subset \R$. Let $\{f_\mu\}_{\mu>0}$ be a family of measurable functions in $\R^n$. If $\liminf_{\mu\to 0} a_\mu\geq a$ and $f$ is a measurable function such that $f\geq \limsups_{\mu\to 0} f_\mu$, then 
\begin{equation}\label{meas lim1b}
m(K \cap(\{f<a\}\setminus \{f_\mu<a_\mu\}))\to 0\quad \text{as $\mu\to 0$.}
\end{equation}
If $\limsup_{\mu\to 0} a_\mu\leq a$ and $f$ is a measurable function such that $f\leq \liminfs_{\mu\to 0} f_\mu$, then
\begin{equation}\label{meas lim2b}
m( K\cap(\{f_\mu\leq a_\mu\}\setminus\{f\leq a\}))\to 0\quad \text{as $\mu\to 0$.}
\end{equation}
Here, we denote by $\limsups_{\mu\to 0} f_\mu$ and $\liminfs_{\mu\to 0} f_\mu$ the half relaxed limits of 
$\{f_\mu\}$ {\rm(}see \cite{GBook, CIL}{\rm)}. 
\end{lem}
\begin{proof}
Let us first show \eqref{meas lim1b}. Since $f\geq \limsups_{\mu\to 0} f_\mu$ and $\liminf_{\mu\to 0}a_\mu\geq a$ hold, for any $x\in \R^n$ satisfying $f(x)<a$, there exists $\mu>0$ such that $f_\delta(x)<a_\delta$ for all $\delta\leq \mu$; that is to say
\[%\label{meas lim1a}
\{f<a\}\subset \liminf_{\delta\to0}\{f_\delta<a_\delta\}:=\bigcup_{\mu>0}\bigcap_{\delta\leq \mu} \{f_\delta<a_\delta\}.
\]
 It follows that 
\[
m\left(\bigcap_{\mu>0}\bigcup_{\delta\leq \mu}K \cap\left(\{f<a\}\setminus \{f_\delta<a_\delta\}\right)\right)=0.
\]
Then for any $\varepsilon>0$, there exists $\mu>0$ small such that 
\[
m\left( \bigcup_{\delta\leq \mu} K \cap \left(\{ f<a\}\setminus \{f_\delta<a_\delta\}\right)\right)<\varepsilon,
\]
which immediately yields
\[
m(K \cap(\{f<a\}\setminus \{f_\mu<a_\mu\}))<\varepsilon.
\]
 We thus complete the proof of  \eqref{meas lim1b}. 

One can show  \eqref{meas lim2b} similarly. In fact, this time our assumptions enable us to  get 
\[%\label{meas lim2a}
\limsup_{\delta\to0}\{f_\delta\le a_\delta\}:=\bigcap_{\mu>0}\bigcup_{\delta\leq \mu}\{f_\delta\leq a_\delta\}\subset \{f\leq a\}.
\]
Then, for any $\varepsilon>0$, 
\[
m\left(\bigcup_{\delta\leq \mu} K \cap\left(\{f_\delta\leq a_\delta\}\setminus \{f\leq a\}\right) \right)<\varepsilon
\]
when $\mu>0$ is taken small. We thus have 
\[
m\left(K \cap(\{f_\mu\leq a_\mu\}\setminus \{f\leq a\}) \right)<\varepsilon
\]
for any $\mu>0$ small, which concludes the proof of \eqref{meas lim2b}. 
\end{proof}

\begin{proof}[Proof of Proposition {\rm\ref{prop:equiv}}]
We only prove the statement for supersolutions. 
Suppose that there exist $\varphi\in C^2(\R^n\times (0, \infty))$ and $(x_0, t_0)$ such that $u-\varphi$ attains a strict local minimum at $(x_0, t_0)$. 
Without loss of generality we may assume  $\|\nabla\varphi\|_{L^\infty}\le R$ for some $R>0$. It is clear that $u$ satisfies the supersolution property if $\nabla \varphi(x_0, t_0)\neq 0$. We thus only consider the case when $\nabla \varphi(x_0, t_0)=0$.

Let 
\[
\Phi(x,y,t):=u(x, t)-\varphi(y, t)-{|x-y|^4\over \varepsilon}.
\] 
By a standard argument, we see that, for any $\varepsilon>0$ small,
\begin{equation}\label{equiv1sup}
\begin{aligned}
&\Phi \ \text{attains a local minimum at some} \ (x_\varepsilon, y_\varepsilon, t_\varepsilon)\in \R^n\times \R^n\times (0, \infty)
\ \text{satisfying} \\
&(x_\varepsilon, y_\varepsilon, t_\varepsilon) \to (x_0, x_0, t_0), \quad u(x_\varepsilon, t_\varepsilon)\to u(x_0, t_0)
\quad\text{as} \ \varepsilon\to0. 
\end{aligned}
\end{equation}
Noting that $y\mapsto \Phi(x_\varepsilon,y,t_\varepsilon)$ takes a local minimum at $y=y_\varepsilon$, 
we have 
\begin{equation}\label{ineq:elementary}
\nabla\varphi(y_\varepsilon, t_\varepsilon)={4\over \varepsilon}|x_\varepsilon-y_\varepsilon|^2(y_\varepsilon-x_\varepsilon), \quad
\nabla^2 \varphi(y_\varepsilon, t_\varepsilon)\leq {4\over \varepsilon}|x_\varepsilon-y_\varepsilon|^2 I+{8\over \varepsilon} (x_\varepsilon-y_\varepsilon)\otimes (x_\varepsilon-y_\varepsilon).  
\end{equation}

We here divide into two cases:
\begin{itemize}
\item Case 1.  $x_\varepsilon\neq y_\varepsilon$ along a sequence $\varepsilon=\varepsilon_j\to 0$;
\item Case 2. $x_\varepsilon=y_\varepsilon$ for all small $\varepsilon>0$. 
\end{itemize}

In Case 1, by \eqref{equiv1sup}, 
\begin{equation*}
(x, t)\mapsto u(x, t)-\varphi(x-x_\varepsilon+y_\varepsilon, t)-{|x_\varepsilon-y_\varepsilon|^4\over \varepsilon}
\end{equation*}
attains a local minimum at $(x_\varepsilon, t_\varepsilon)$. Thus we can adopt \eqref{f-super} to obtain
\begin{equation}\label{equiv7sup}
\varphi_t(y_\varepsilon, t_\varepsilon)+F(u(x_\ep, t_\ep), \nabla \varphi(y_\varepsilon, t_\varepsilon), \nabla^2\varphi(y_\varepsilon, t_\varepsilon),K\cap\{u(\cdot, t_\varepsilon)\leq u(x_\varepsilon, t_\varepsilon)\})\geq 0.
\end{equation}
%Noticing that 
Due to \eqref{equiv1sup} and the lower semicontinuity of $u$, by Lemma \ref{lem meas} we have
\begin{equation}\label{equiv8sup}
m\big(K\cap(\{u(\cdot, t_\varepsilon)\leq u(x_\varepsilon, t_\varepsilon)\}\setminus \{u(\cdot, t_0)\leq u(x_0, t_0)\})\big)\to 0\quad \text{as $\varepsilon\to 0$. }
\end{equation}

Set $a_\ep:=u(x_\ep, t_\ep)$, $\eta_\varepsilon:=\nabla\varphi(y_\varepsilon,t_\varepsilon)\not=0$, 
$Y_\varepsilon:=\nabla^2\varphi(y_\varepsilon,t_\varepsilon)$, 
$A_\varepsilon:=K\cap\{u(\cdot,t_\varepsilon)\le u(x_\varepsilon, t_\varepsilon)\}$, and 
$B:=K\cap\{u(\cdot,t_0)\le u(x_0, t_0)\}$. 
By \eqref{equiv7sup}, (F3), and (F4), 
\begin{align*}
0\le&\, 
\varphi_t(y_\varepsilon, t_\varepsilon)+F(a_\ep, \eta_\varepsilon, Y_\varepsilon, A_\varepsilon)\\
\le&\, \varphi_t(y_\varepsilon, t_\varepsilon)+F(a_\ep, \eta_\varepsilon, Y_\varepsilon, B)
+F(a_\ep, \eta_\varepsilon, Y_\varepsilon, A_\varepsilon)-F(a_\ep, \eta_\varepsilon, Y_\varepsilon, A_\varepsilon\cap B)\\
\le&\, 
\varphi_t(y_\varepsilon, t_\varepsilon)+F(a_\ep, \eta_\varepsilon, Y_\varepsilon, B)+\omega_R(m(A_\varepsilon\triangle (A_\varepsilon\cap B)))\\
=&\, 
\varphi_t(y_\varepsilon, t_\varepsilon)+F(a_\ep, \eta_\varepsilon, Y_\varepsilon, B)+\omega_R(m(A_\varepsilon\setminus B)). 
\end{align*}
By \eqref{equiv8sup}, sending $\varepsilon\to0$ yields 
\begin{equation}\label{equiv5sup}
\varphi_t(x_0, t_0)+F^\ast(u(x_0, t_0), 0, \nabla^2 \varphi(x_0, t_0), B)\geq 0.
\end{equation}

In Case 2, we see that 
$u-\psi_\varepsilon$ attains a local minimum at $(x_\varepsilon, t_\varepsilon)$, where $\psi_\varepsilon$ is given by
\begin{equation*}
\psi_\varepsilon(x, t):=\varphi(y_\varepsilon, t)+{|x-y_\varepsilon|^4\over \varepsilon}.
\end{equation*}
Since $x_\varepsilon=y_\varepsilon$, it is clear that 
\begin{equation*}
\nabla\psi_\varepsilon(x_\varepsilon, t_\varepsilon)=0, \quad \nabla^2 \psi_\varepsilon(x_\varepsilon, t_\varepsilon)=0.
\end{equation*}
It then follows from \eqref{f-singular} that 
\[ 
\varphi_t(y_\varepsilon, t_\varepsilon)
+ \mu(u(x_\ep, t_\ep))=(\psi_\varepsilon)_t (x_\varepsilon, t_\varepsilon) + \mu(u(x_\ep, t_\ep)) \geq 0. 
\]
Sending $\varepsilon\to 0$ yields 
\begin{equation} \label{limit-equivsup}
\varphi_t(x_0, t_0) + \mu(u(x_0,t_0)) \geq 0.
\end{equation}

By \eqref{ineq:elementary} we have 
\[
\nabla\varphi(y_\varepsilon, t_\varepsilon)=0, \quad \nabla^2\varphi(y_\varepsilon, t_\varepsilon)\leq 0.
\]
Sending $\varepsilon\to 0$, we obtain 
$\nabla\varphi(x_0, t_0)=0$, and $\nabla^2 \varphi(x_0, t_0)\leq 0$.
Combining these with \eqref{limit-equivsup}
%\[
%\varphi_t(x_0, t_0)+\mu(u(x_0, t_0))\geq 0
%\]
 and applying (F6), we are led to  \eqref{equiv5sup} again. 
\end{proof}

\section{Comparison Principle}\label{sec:cp}
In this section, we present a comparison principle for \eqref{E1}.  %for our later applications. 
\begin{thm}[Comparison principle]\label{thm:comp}
Assume that {\rm(F1)--(F6)} hold. 
Let $u\in \USC(\R^n\times [0, \infty))$ and $v\in \LSC(\R^n\times [0, \infty))$ be, respectively, a subsolution and a supersolution to \eqref{E1}. 
Assume in addition that for any $T>0$, there exists $L_T>0$ such that 
\begin{equation}\label{growth comp}
u(x, t)\leq L_T(|x|+1), \quad v(x, t)\geq -L_T(|x|+1)\quad \text{for all $(x, t)\in \R^n\times [0, T]$.}
\end{equation}
If there exists a modulus of continuity $\omega_0$ such that
\begin{equation}\label{initial comp}
u(x, 0)-v(y, 0)\leq \omega_0(|x-y|)\quad \text{for all $x, y\in \R^n$,}
\end{equation}
then $u\leq v$ holds in $\R^n\times [0, \infty)$.
%\[
%u(x,t)-v(y,t))\leq \omega(|x-y|) %\le \sup_{\R^n}(u(\cdot,0)-v(\cdot,0)) 
%quad\text{for all $x, y\in \R^n$ and $t\ge0$.}  
%\]
\end{thm}

The following result is an adaptation of \cite[Proposition 2.3]{GGIS}. 
\begin{prop}[Growth estimate]\label{prop growth}
Assume the same assumptions as in Theorem {\rm\ref{thm:comp}}. For any fixed $T>0$ and any $L>L_{T}$ large, there exists $M>0$ such that 
\[
u(x, t)-v(y, t)\leq L|x-y|+M(1+t)\quad \text{for all $x, y\in \R^n$ and $t\in [0, T)$.}
\]
\end{prop}
\begin{proof}
%\ma{
%Our proof is similar to that of \cite[Proposition 2.3]{GGIS}. 
%Prop 3.2の直前に17のadaptationとあるからこの文はちょっと重複？？
%}
Note that \eqref{initial comp} yields the existence of $C_0>0$ such that 
\begin{equation}\label{initial comp2}
u(x, 0)-v(y, 0)\leq C_0(|x-y|+1)\quad \text{for all $x, y\in \R^n$.}
\end{equation}
Take $L>\max\{L_{T}, C_0\}$. Our goal is to show that 
\begin{equation}\label{ineq:first-step}
u(x, t)-v(y, t)-\psi(x, y)-M(1+t)\leq 0\quad \text{for $x, y\in \R^n$, $t\in [0, T)$}
\end{equation}
for $M>0$ sufficiently large, where 
\[
\psi(x, y)=L(|x-y|^2+1)^{1\over 2}.
\]
Suppose that \eqref{ineq:first-step} fails to hold for any arbitrarily large $M>0$. 
Then, we may assume that $M>C_0$, and 
there exist $\hat{x}, \hat{y}\in \R^n, \hat{t}\in [0, T)$ such that 
\begin{equation}\label{rough1}
u(\hat{x}, \hat{t})-v(\hat{y}, \hat{t})-\psi(\hat{x}, \hat{y})-M(1+\hat{t})>0.
\end{equation}
Set, for $\ep, \lambda>0$ small and $R>0$ large, 
\begin{align*}
\Psi_\ep(x, t, y, s)=&\; u(x, t)-v(y, s)-\psi(x, y)-L (g_R(x) + g_R(y)) \\
&\; -{(t-s)^2\over 2\ep}-M(1+t) -\frac{\lambda}{T-t},
\end{align*}
where $g_R\in C^2(\R^n)$ is a nonnegative function satisfying 
\begin{align*}
&g_R(x)=0 \quad \text{for $|x|<R$,} \quad
{g_R(x)\over |x|}\to 1 \quad \text{as $|x|\to \infty$,}\\
&
\sup\{|\nabla g_R(x)|+|\nabla^2 g_R(x)|: x\in \R^n,\  R>0\}<\infty.
\end{align*}
It follows from \eqref{rough1} that $\Psi_\ep(\hat{x}, \hat{t}, \hat{y}, \hat{t})>0$ if we take $R>|\hat{x}|, |\hat{y}|$ and $\lambda > 0$ small depending only on $M$. %takes a positive value in $(\R^{n}\times [0, T])^2$, 

By \eqref{growth comp}, we see that $\Psi_\ep$ attains a positive maximum in $(\R^{2}\times [0, T))^2$ at $(x_\ep, t_\ep, y_\ep, s_\ep)$ for $R>0$ and $M>0$ large and for $\ep>0$ small.  In fact,  $x_\ep$ and $y_\ep$ are bounded uniformly with respect to $\ep$. Moreover, we have $t_\ep, s_\ep\to t_0$ for some $t_0\in [0, T)$ as $\ep\to 0$.
 In view of the upper semicontinuity of $u$ and lower semicontinuity of $v$ as well as \eqref{initial comp2}, we deduce that $t_0\neq 0$ and thus $t_\ep, s_\ep>0$ for all $\ep>0$ small.
%and show that $\Psi_\vep$ does not admit a positive 

Since $u$ and $v$ are, respectively, a subsolution and a supersolution of \eqref{E1}, we obtain
\begin{align*}
&{t_\ep-s_\ep\over \ep}+M+ \frac{\lambda}{(T-t_\ep)^2} + F_\ast(u(x_\ep, t_\ep), p_1, X_1, K\cap \{u(\cdot, t_\ep)<u(x_\ep, t_\ep)\})\leq 0,\\
&{t_\ep-s_\ep\over \ep}+F^\ast(v(y_\ep, s_\ep), p_2, X_2, K\cap \{v(\cdot, t_\ep)\leq v(y_\ep, t_\ep)\})\geq 0,
\end{align*}
where 
\[
p_1=L(|x_\ep-y_\ep|^2+1)^{-{1\over 2}}(x_\ep-y_\ep) + L \nabla g_R(x_\ep), \] \[p_2=L(|x_\ep-y_\ep|^2+1)^{-{1\over 2}}(x_\ep-y_\ep) - L \nabla g_R(y_\ep),
\]
\[
X_1=L(|x_\ep-y_\ep|^2+1)^{-{1\over 2}} I-L(|x_\ep-y_\ep|^2+1)^{-{3\over 2}}(x_\ep-y_\ep)\otimes (x_\ep-y_\ep) +L\nabla^2 g_R(x_\ep),
\]
\[
X_2=-L(|x_\ep-y_\ep|^2+1)^{-{1\over 2}} I+L(|x_\ep-y_\ep|^2+1)^{-{3\over 2}}(x_\ep-y_\ep)\otimes (x_\ep-y_\ep)- L\nabla^2 g_R(y_\ep).
\]
Since the boundedness of $p_1, p_2, X_1, X_2$ depends only on $L$, taking the difference between the viscosity inequalities above and applying (F2), we have $C_L>0$ such that $M\leq C_L$, which is a contradiction to the arbitrariness of $M>0$. 
\end{proof}

Let us now proceed to the proof of Theorem \ref{thm:comp}.
\begin{proof}[Proof of Theorem {\rm\ref{thm:comp}}]
%Without loss of generality, we may assume 
%\[
%\sup_{\R^n}(u(\cdot,0)-v(\cdot,0))\le0.
%\] 
Assume by contradiction that 
$\sup_{\R^n\times[0,T)}(u-v)=:\theta>0$. %\fbox{\bl{replace $[0,T]$ to $[0,T)$?}}
Then, there exists $\lambda>0$ such that 
\[
\sup_{(x,t)\in\R^n\times[0,T)}\left\{u(x,t)-v(x,t)-\frac{\lambda}{T-t}\right\}>\frac{3\theta}{4}. 
\]
There exists $(x_1,t_1)\in\R^n\times[0,T)$ such that 
$u(x_1,t_1)-v(x_1,t_1)-\lambda/(T-t_1)>\theta/2$. 
Noting that $\sup_{\R^n}(u(\cdot,0)-v(\cdot,0))\le0$, we have $t_1>0$. 

Define 
\[
\Phi(x,y,t):=u(x,t)-v(y,t)-\frac{|x-y|^4}{\varepsilon^4}-\alpha(|x|^2+|y|^2)-\frac{\lambda}{T-t}
\]
for $\varepsilon, \alpha>0, \lambda>0$.  It is then clear that there exists $\alpha_0>0$ small such that 
\[
\sup_{(x, y, t)\in \R^{2n}\times [0, T)}\Phi(x, y, t) >{\theta\over 4}
\]
for all $0<\alpha<\alpha_0$ and $\ep>0$ small. 
The growth condition \eqref{growth comp} implies that $\Phi$ attains a maximum at some 
$(x_{\varepsilon, \alpha}, y_{\varepsilon, \alpha}, t_{\varepsilon, \alpha})\in\R^n\times[0,T)$. 
We write $(\tilde{x}, \tilde{y}, \tilde{t})$ for $(x_{\varepsilon, \alpha}, y_{\varepsilon, \alpha}, t_{\varepsilon, \alpha})$ by abuse of notations.  It follows that 
\[
\begin{aligned}
&\frac{|\tilde{x}-\tilde{y}|^4}{\varepsilon^4}+\alpha(|\tilde{x}|^2+|\tilde{y}|^2)\\
&\leq u(\tilde{x}, \tilde{t})-v(\tilde{y}, \tilde{t})- u(x_1, t_1)+v(x_1, t_1)+2\alpha|x_1|^2+{\lambda\over T-t_1}-{\lambda\over T-\tilde{t}}.
\end{aligned}
\]
In view of Proposition \ref{prop growth}, we have 
\[
\begin{aligned}
&\frac{|\tilde{x}-\tilde{y}|^4}{\varepsilon^4}+\alpha(|\tilde{x}|^2+|\tilde{y}|^2) \leq L(|\tilde{x}-\tilde{y}|+1)+M(\tilde{t}+1)\\%+\alpha(|\tilde{x}|^2+|\tilde{y}|^2)\\
&\qquad - u(x_1, t_1)+v(x_1, t_1)+2\alpha|x_1|^2+{\lambda\over T-t_1}-{\lambda\over T-\tilde{t}}.
\end{aligned}
\] 
It follows that 
\[
\frac{|\tilde{x}-\tilde{y}|^4}{\varepsilon^4}-L|\tilde{x}-\tilde{y}|+\alpha(|\tilde{x}|^2+|\tilde{y}|^2)\le C
\] 
for some $C\ge0$ which is independent of $\varepsilon, \alpha$, which implies that 
\begin{equation}\label{conv:alpha}
\alpha(|\tilde{x}|+|\tilde{y}|)\to0 \ \text{as} \ \alpha\to0 \ 
\text{for any} \ \ep>0, 
\quad
\text{and} 
\quad
\sup_{0<\alpha<\alpha_0} |\tilde{x}-\tilde{y}|\to 0\quad \text{as $\ep\to 0$}. 
\end{equation}
Hence, there exists $\ep_0>0$ such that
\[%\label{close}
\omega_0(|\tilde{x}-\tilde{y}|)\leq {\theta/4}
\]
uniformly for all $0<\ep<\ep_0$ and $0<\alpha<\alpha_0$, where $\omega_0$  is the modulus of continuity appearing in \eqref{initial comp}. 

%We thus claim that 
%\begin{equation}\label{time nonzero}
%\liminf_{\alpha\to 0} \tilde{t}>0
%\end{equation}
%if we take $0<\ep<\ep_0$.
%Suppose that this fails to hold and there exists a sequence $\alpha_i\to 0$ such that the corresponding maximizer $\tilde{t}\to 0$. Then by \eqref{close} and \eqref{initial comp}, we get
%\[
%\limsup_{\alpha_i\to 0} (u(\tilde{x}, \tilde{t})-v(\tilde{y}, \tilde{t}))\leq {\theta/4}
%\]
%for any $0<\ep<\ep_0$.
On the other hand, we have%, for all $0<\alpha<\alpha_0$,
\[
u(\tilde{x}, \tilde{t})-v(\tilde{y}, \tilde{t})\geq \Phi(\tilde{x}, \tilde{y}, \tilde{t})>{\theta\over 4}.
\]
%Hence, we conclude that \eqref{time nonzero} holds for all $0<\ep<\ep_0$. As an immediate consequence, we obtain
It follows that  $\tilde{t}>0$ for any $0<\alpha<\alpha_0$ and $0<\ep<\ep_0$. In what follows, we fix $0<\ep<\ep_0$.

By the Crandall-Ishii lemma \cite{CIL}, for all $\rho>0$, there exists 
$(h, p+2\alpha \tilde{x}, X)\in \overline{P}^{2,+}u(\tilde{x},\tilde{t})$, 
$(k, p-2\alpha \tilde{y}, Y)\in \overline{P}^{2,-}v(\tilde{y},\tilde{t})$ such that 
\begin{align}
&h-k=\frac{\lambda}{(T-\tilde{t})^2}, \ p=\frac{4|\tilde{x}-\tilde{y}|^2(\tilde{x}-\tilde{y})}{\varepsilon^4}, \nonumber\\ 
&\begin{pmatrix}
X &0 \\
0 & -Y
\end{pmatrix}
\le 
\begin{pmatrix}
Z &-Z \\
-Z & Z
\end{pmatrix}
+
\rho
\begin{pmatrix}
Z &-Z \\
-Z & Z
\end{pmatrix}^2, \nonumber
\end{align}
where we denote by $P^{\pm}u(x,t)$ the semijets of $u$ at $(x,t)$ (see \cite{CIL, GBook} for the definition), 
and 
\[
Z:=\frac{4|\tilde{x}-\tilde{y}|^2}{\varepsilon^4}(I+2(\tilde{x}-\tilde{y})\otimes (\tilde{x}-\tilde{y}))+
2\alpha 
\begin{pmatrix}
I & I\\
I & I
\end{pmatrix}. 
\]

Here, fix $\varepsilon>0$ small enough so that $\tilde{t}>0$. We divide into two cases:  
\begin{align*}
&\text{Case 1}. \ \liminf_{\alpha\to 0} |\tilde{x}-\tilde{y}|>0,  \\
%\text{There exists} \ \alpha_i\to0 \ \text{such that} \  \tilde{x}\not=\tilde{y} \ \text{for} \ \forall\, i\in\mathbb{N}, \ \text{and} \ \liminf_{\alpha_i\to0}|\tilde{x}-\tilde{y}|\not=0. \\
&\text{Case 2}.  \ \liminf_{\alpha\to 0} |\tilde{x}-\tilde{y}|=0, \ 
\text{i.e.,} \ \exists \alpha_i\to0 \ \text{such that} \ % \tilde{x}\not=\tilde{y} \ \text{for} \ \forall\, i\in\mathbb{N}, \ \text{and} \ 
\lim_{\alpha_i\to0}|\tilde{x}-\tilde{y}|=0. 
%&\text{Case 3}. \ 
%\tilde{x}=\tilde{y} \ \text{for all small} \ \alpha>0.  \quad \textcolor{blue}{\text{Is Case 3 actually covered by Case 2？}}
\end{align*}

Let us consider Case 1 first. Since $\Phi(x,x,\tilde{t})\le \Phi(\tilde{x},\tilde{y},\tilde{t})$ for all $x\in\R^n$, we have 
\begin{equation}\label{ineq:key}
u(x,\tilde{t})-u(\tilde{x},\tilde{t})\le v(x,\tilde{t})-v(\tilde{y},\tilde{t})-\frac{|\tilde{x}-\tilde{y}|^4}{\varepsilon^4}+2\alpha|x|^2-\alpha(|\tilde{x}|^2+|\tilde{y}|^2). 
\end{equation}
%We divide into two cases furthermore: \textcolor{blue}{There is no need to further divide the argument?}
%\begin{itemize}
%\item Case 1(a). $2\alpha|x|^2-\alpha(|\tilde{x}|^2+|\tilde{y}|^2)<0$ for all $x\in K$;
%\item Case 1(b). There exists  $x_0\in K$ such that  $2\alpha|x_0|^2-\alpha(|\tilde{x}|^2+|\tilde{y}|^2) \ge 0$. 
%\end{itemize}

Since $K$ is bounded, we have 
\[
\begin{aligned}
&\liminf_{\alpha\to 0}\left(-\frac{|\tilde{x}-\tilde{y}|^4}{\varepsilon^4}+2\alpha \max_{x\in K} |x|^2-\alpha(|\tilde{x}|^2+|\tilde{y}|^2)\right)\\
\leq &\liminf_{\alpha\to 0}\left(-\frac{|\tilde{x}-\tilde{y}|^4}{\varepsilon^4}+2\alpha \max_{x\in K} |x|^2\right)<0,
\end{aligned}
\]
which implies that, for all $x\in K$ and $\alpha>0$ small,
\[
-\frac{|\tilde{x}-\tilde{y}|^4}{\varepsilon^4}+2\alpha |x|^2-\alpha(|\tilde{x}|^2+|\tilde{y}|^2)<0.
\]
%In Case 1(a)
By \eqref{ineq:key},  we thus have $u(x,\tilde{t})-u(\tilde{x},\tilde{t})<v(x,\tilde{t})-v(\tilde{y},\tilde{t})$ for all $x\in K$, which implies 
\begin{equation}\label{set implication}
K\cap\{v(\cdot,\tilde{t})\leq v(\tilde{y},\tilde{t})\}\subset K\cap\{u(\cdot,\tilde{t})<u(\tilde{x},\tilde{t})\}. 
\end{equation}

Moreover, noticing that $|p|$ is bounded away from $0$ uniformly in  $\alpha$, %taking a subsequence if necessary, we may assume 
%$p\to p^\varepsilon$ as $\alpha_i\to0$ for some $p^\varepsilon\in\R^n$ with $p^\varepsilon\neq 0$. 
we have 
%$p_1\neq 0$ and $p_2\neq 0$ for all $\alpha>0$ small, where we set
\begin{equation}\label{gradients}
p_1:=p+2\alpha\tilde{y}\neq 0, \quad p_2:=p-2\alpha\tilde{x}\neq 0 
\quad\text{for all $\alpha>0$ small}. 
\end{equation}
%$p+2\alpha\tilde{x}\neq 0$ and $p-2\alpha\tilde{y}\neq 0$  
Since $u$ and $v$, respectively,  are a viscosity subsolution and supersolution to \eqref{E1}, we get
\begin{align*}
&h+F(u(\tilde{x}, \tilde{t}), p_1,  X, K\cap\{u(\cdot,\tilde{t})<u(\tilde{x},\tilde{t})\})\leq 0, \\
&k+F(v(\tilde{y}, \tilde{t}), p_2, Y, K\cap\{v(\cdot,\tilde{t})\leq v(\tilde{y},\tilde{t})\})\geq 0. 
 \end{align*}
 Since \eqref{set implication} and $v(\tilde{y}, \tilde{t})<u(\tilde{x}, \tilde{t})$ hold, applying (F1) and (F4) to the second inequality above yields 
\[
k+F(u(\tilde{x}, \tilde{t}), p_2, Y, K\cap\{u(\cdot,\tilde{t})< u(\tilde{x},\tilde{t})\})\geq 0.   
\]
Using (F5), we then obtain 
\[
\begin{aligned}
\frac{\lambda}{T^2} & \le 
\frac{\lambda}{(T-\tilde{t})^2} =h-k\\
&\le   F(u(\tilde{x}, \tilde{t}), p_2, Y, K\cap\{u(\cdot,\tilde{t})< u(\tilde{x},\tilde{t})\})-
F(u(\tilde{x}, \tilde{t}), p_1, X, K\cap\{u(\cdot,\tilde{t})<u(\tilde{x},\tilde{t})\})\\
&\le  \omega\left(\frac{|Z||p_1-p_2|}{\min\{|p_1|, |p_2|\}}+|p_1-p_2|+\alpha\right)+O(\rho)\\
&\le  \omega\left(\frac{2\alpha|Z|(|\tilde{x}|+|\tilde{y}|)}{\min\{|p_1|, |p_2|\}}+\alpha(2|\tilde{x}|+2|\tilde{y}|+1)\right)+O(\rho). 
\end{aligned}
\]
%\textcolor{blue}{(Maybe $\min\{|p_1|, |p_2|\}$ instead of $|p_1|+|p_2|$?)} 
Note that 
\[
|Z|\le c, \quad 
\min\{|p_1|, |p_2|\}\ge \frac{1}{c}
\]
for some $c>0$ which is independent of $\alpha$. Sending $\rho\to0$, $\alpha\to0$ in this order and using \eqref{conv:alpha} yield 
$\lambda/T^2\le 0$, which is a contradiction.

Let us turn to Case 2. In this case, fixing $\rho>0$ small, we have 
\[
|p|\to0, \ X\to X^\varepsilon, \ Y\to Y^\varepsilon \quad\text{as} \ \alpha_i\to0, 
\]
and 
\[
\begin{pmatrix}
X^\varepsilon & 0\\
0 & - Y^\varepsilon 
\end{pmatrix}
\le 0, 
\]
which implies $X^\varepsilon \le 0\le Y^\varepsilon$. % For simplicity of notation, below we index the subsequence still by $\alpha$ instead of $\alpha_i$.  
Let us consider the subsequence along $\alpha_i$. Let $p_1, p_2$ be as in \eqref{gradients} above. 
Then we can adopt the definition of subsolutions and supersolutions to get
\[
0\ge   
h+F_\ast(u(\tilde{x}, \tilde{t}), p_1, X, K\cap\{u(\cdot,\tilde{t})<u(\tilde{x},\tilde{t})\})
\ge 
h+F_\ast(u(\tilde{x}, \tilde{t}), p_1, X, \emptyset), 
\]
\[
0\le k+F^\ast(v(\tilde{y}, \tilde{t}), p_2, Y, K\cap\{v(\cdot,\tilde{t})<v(\tilde{y},\tilde{t})\})
\le  k+F^\ast(v(\tilde{y}, \tilde{t}), p_2, Y, K). 
\]
%\begin{align*}
%0\ge &\,  
%a+F_\ast(u(\tilde{x}, \tilde{t}), p+2\alpha\tilde{x}, X, K\cap\{u(\cdot,\tilde{t})<u(\tilde{x},\tilde{t})\})\\
%\ge &\, 
%a+F_\ast(u(\tilde{x}, \tilde{t}), p+2\alpha\tilde{x}, X, \emptyset), 
%\end{align*}
%\begin{align*}
%0\le &\, 
%b+F_\ast(p-2\alpha\tilde{y}, Y, K\cap\{v(\cdot,\tilde{t})<v(\tilde{y},\tilde{t})\})\\
%\le &\, 
%b+F_\ast(p-2\alpha\tilde{y}, Y, K). 
%\end{align*}
Thus,  we are led to
\[
\frac{\lambda}{T^2}\le 
\frac{\lambda}{(T-\tilde{t})^2}=h-k
\le 
F^\ast(v(\tilde{y}, \tilde{t}), p_2, Y,  K)-F_\ast(u(\tilde{x}, \tilde{t}), p_1, X, \emptyset).
\]
Sending $\alpha\to0$ and applying (F1), we get 
\[
\begin{aligned}
\frac{\lambda}{T^2}&\le \limsup_{\alpha_i\to 0}\left(F^\ast(v(\tilde{y}, \tilde{t}), 0, Y_\ep, K)-F_\ast(u(\tilde{x}, \tilde{t}), 0, X_\ep, \emptyset)\right)\\
&\le \limsup_{\alpha_i\to 0}\left(F^\ast(u(\tilde{x}, \tilde{t}), 0, 0,  K)-F_\ast(u(\tilde{x}, \tilde{t}), 0, 0, \emptyset)\right)
\end{aligned}
\]
Using (F6) yields $\lambda/T^2\le 0$, which is a contradiction.  
\end{proof}

%\begin{thm}[Comparison principle]\label{thm:comp}
%Assume that (F1)--(F7) hold. %and $u\in\USC(\mathbb{R}^n\times[0,T])$, $v\in\LSC(\mathbb{R}^n\times[0,T])$ be bounded on $\mathbb{R}^n\times[0,T]$. 
%Let $u$ and $v$ be respectively a subsolution and a supersolution to \eqref{E1}  %Assume in addition that there exists a compact set $\cK\subset \R^n$ such that $u\leq v$ in $(\R^n\setminus \cK) \times [0, \infty)$. 
%such that $u(\cdot, t), v(\cdot, t)\in \ocC(\R^n)$ holds for any $t>0$. Assume also that 
%\begin{equation}\label{cond1}
%\lim_{\delta\to 0}\sup\{u(x, t)-v(y, s): |x|, |y|\geq 1/\delta,\ 0\leq t, s\leq T, |x-y|+|t-s|\leq \delta\}\leq 0
%\end{equation}
%for any $T>0$, and 
%\begin{equation}\label{cond2}
%\lim_{\delta\to 0}\sup\{u(x, t)-v(y, s): |x-y|\leq \delta,\ 0\leq t, s\leq \delta\}\leq 0.
%\end{equation}
%Then $u\le v$ in $\mathbb{R}^n\times[0,\infty)$. 
%\end{thm}

By using the stability result \cite[(P2)]{Sl}, we can adapt the standard ``bump-up" argument in Perron's method 
(see \cite{CIL, GBook} for instance) to \eqref{E1}. 
Assuming that $u_0$ is uniformly continuous, we can prove the existence of the unique viscosity solution. 
In addition, if there exists an appropriate subsolution as below,  then by the comparison principle, we obtain the existence of a unique viscosity solution that satisfies the positivity and coercivity conditions \eqref{ineq:positive} and \eqref{coercive}.
\begin{enumerate}
\item[(I)] There exists a function $\phi\in C(\mathbb{R}^n \times [0,\infty))$ such that
\begin{itemize}
\item[(i)] $\phi(\cdot, t) \in \UC (\mathbb{R}^n)$ for any $t \ge 0$, 
\item[(ii)] $u_0 \ge \phi(\cdot , 0)$ in $\mathbb{R}^n$, 
\item[(iii)] $\phi \ge c_0$ in $\mathbb{R}^n\times [0, \infty)$ for some $c_0>0$.
\item[(iv)] $\phi$ is coercive in space, that is, 
\[
\inf_{|x|\geq R,\ t \le T}\phi(x,t) \to \infty\quad \text{as $R\to \infty$ for any $T \ge 0$}.
\]
\item[(v)] $\phi$ is a viscosity subsolution of \eqref{E1}.
%\[
%\phi_t(x,t) +F(\phi(x,t), \nabla \phi(x,t), \nabla^2 \phi(x,t), K\cap \{\phi(\cdot,t)< \phi(x,t)\})\leq 0\quad \text{in $\R^n \times (0,\infty)$}
%\]
%in the viscosity sense. 
\end{itemize}
\end{enumerate}

\begin{thm}[Existence]\label{thm:exist}
Assume that {\rm(F1)--(F6)} hold. Let $u_0\in\UC(\R^n)$. Then there exists a unique solution $u$ of \eqref{E1}, \eqref{initial} that satisfies \eqref{uc-preserving}. 
Moreover, if the additional assumption {\rm(I)} holds, then the unique solution $u$ also satisfies \eqref{ineq:positive} and \eqref{coercive}. 
% positive in $\R^n\times [0, \infty)$ and satisfies the coercivity condition \eqref{coercive}.
%\[
%u(x, t)\geq \phi(x)+\nu t\quad \text{for  all $(x, t)\in \R^n\times [0, \infty)$}
%\]
%and, in particular, .
\end{thm}
%The second half of the statements above is a result of  due to the comparison principle. 

%\begin{thm}[Existence]\label{thm:exist}
%Assume that (F1)--(F7) hold. Assume that $u_0\in\ocC(\mathbb{R}^n)$ is uniformly continuous. 
%Then there exists a unique \underline{uniformly continuous} solution $u$ of \eqref{E1}--\eqref{initial} such that $u(\cdot, t)\in \ocC(\R^n)$ for all $t\geq 0$. 
%\end{thm}

%$u(\cdot, t)\in \ocC(\R^n)$ for all $t>0$.
%satisfying the following conditions:
%\begin{enumerate}
%\item[(i)] $u\leq a$ in $\R^n\times [0, T)$;
%\item[(ii)] There exists a compact set $\cK$ in $\R^n$ such that $u\equiv a$ in $(\R^n\setminus \cK)\times [0, T)$. 
%\end{enumerate}  

%By the comparison principle and the fact that constants are solutions of \eqref{E1}, we see that the unique solution $u$ is bounded in $\R^n\times [0, \infty)$ provided that $u_0\in \cC_a(\R^n)$. 

%%%%%%%%%%%%
\begin{comment}
\begin{prop}[Geometric property]\label{prop:geometric}
Assume that {\rm(F1)--(F5)} hold. Let $g:\mathbb{R}\to\mathbb{R}$ be a nondecreasing continuous function. 
If $u$ is a subsolution {\rm(}resp., a supersolution, a solution{\rm)} of \eqref{E1}, then 
$g\circ u$ is also a subsolution {\rm(}resp., a supersolution, a solution{\rm)} of \eqref{E1}.
\end{prop}

We omit the proof since the proof is similar to that of the local case (see \cite[Theorem 4.2.1]{GBook} by using 
the stability result for nonlocal case \cite[(P2)]{Sl}. 
\end{comment}
%%%%%%%%%%%%

\section{Quasiconvexity preserving}\label{sec:quasiconvex}
This section is devoted to proving our main result, Theorem \ref{thm:main}. 
 Fix arbitrarily $\lambda\in (0, 1)$. For $u\in C(\R^n\times [0, \infty))$, let $u_{\star, \lambda}$ be given by \eqref{eq:envelop2-lam}. 
%Defining the quasi-convex envelop for $u\in C(\R^n\times [0, \infty))$ by 
%\begin{equation}\label{eq:envelop2}
%u_\star(x, t)=\inf\bigg\{\max \{u(y, t), u(z, t)\}: x=\lambda y+(1-\lambda)z \ \text{for some $\lambda\in (0, 1)$}\bigg\},
%\end{equation}
Our goal is to show that 
\begin{equation}\label{main-goal2}
u_{\star, \lambda}(x, t)=u(x, t)\quad \text{for all } (x,t)\in\R^n\times[0, \infty). 
\end{equation}
Theorem \ref{thm:main} follows immediately, since $u$ is quasiconvex in space if and only if \eqref{main-goal2} holds for all $\lambda\in (0, 1)$. 
By the definition of $u_{\star, \lambda}$, it is clear that $u_{\star, \lambda}\leq u$ in $\R^n\times [0, \infty)$. It thus suffices to prove the reverse inequality. To this end, we first approximate $u_{\star, \lambda}$ %is a supersolution to \eqref{E1}
%approximation below.  
% Fix any $\lambda\in (0, 1)$. Let $u_{\star, \lambda}$ be given by \eqref{eq:envelop2-lam}.
% It is clear that 
% \begin{equation}
%u_\star(x, t)= \inf_{\lambda\in (0, 1)} u_{\star, \lambda}(x, t)
% \end{equation}
% for any $(x, t)\in \R^n\times (0, \infty)$. 
% By adding a constant, we can assume $u_0\ge c_0$ in $\R^n$ 
%for some $c_0>0$ without loss of generality, and 
%by Theorems \ref{thm:comp}, henceforth we always
%We define
 via the power convex envelope function $u_{q,\lambda}$ ($q>1$) given by %for a fixed $\lambda\in(0,1)$ by 
\begin{equation}\label{eq:q-envelop}
 u_{q, \lambda}(x, t)=\inf\left\{\left(\lambda u(y, t)^q+(1-\lambda) u(z, t)^q\right)^{{1\over q}}: \lambda y+(1-\lambda)z=x\right\}
\end{equation}
for all $(x, t)\in \R^n\times [0, \infty)$.

\begin{prop}[Approximation by power convex envelope]\label{prop:converge}
Let $u\in C(\R^n\times [0, \infty))$ satisfy \eqref{ineq:positive} for some $c_0>0$. Fix $\lambda \in (0, 1)$. Let $u_{\star, \lambda}$ and $u_{q, \lambda}$ be given respectively by \eqref{eq:envelop2-lam} and \eqref{eq:q-envelop}.
Then %\textcolor{red}{$u_{q, \lambda},\ u_{\star, \lambda}\in \LSC(\R^n\times [0, \infty))$?} and 
$u_{q, \lambda}\to u_{\star, \lambda}$ locally uniformly in $\R^n\times [0, \infty)$ as $q\to \infty$.
\end{prop}

\begin{proof}
By definition, it is clear that $u_{q, \lambda}\leq u_{\star, \lambda}$ in $\R^n\times [0, \infty)$. Let $(x,t)\in\R^n\times [0, \infty)$. 
For any $j\in\mathbb{N}$, there exist $y_j, z_j\in\R^n$ such that 
\[
x=\lambda y_j+(1-\lambda)z_j, \quad
\big(\lambda u(y_j,t)^q+(1-\lambda)u(z_j,t)^q\big)^{1\over q}\le u_{q,\lambda}(x,t)+\frac{1}{j}. 
\]

It follows from the positiveness of $u$ that 
\begin{equation}\label{eq:upper bdd}
\begin{aligned}
 \max\left\{u(y_j, t), u(z_j, t)\right\}
 &\leq \left({1\over \min\{\lambda, 1-\lambda\}}\right)^{1\over q} (u_{q, \lambda} (x, t)+1)\\
 &\leq \left({1\over \min\{\lambda, 1-\lambda\}}\right)^{1\over q} (u_{\star, \lambda} (x, t)+1)
 \end{aligned}
\end{equation}
Moreover, we have 
\begin{equation*}
u_{\star, \lambda}(x, t)-u_{q, \lambda}(x, t)\leq \max\left\{u(y_j, t), u(z_j, t)\right\}- \left(\lambda u(y_j, t)^q+(1-\lambda) u(z_j, t)^q\right)^{{1\over q}}+\frac{1}{j}.
\end{equation*}
This yields
\begin{equation}\label{attain-q app1}
\begin{aligned}
&u_{\star, \lambda}(x, t)-u_{q, \lambda}(x, t)\\
&\leq {1\over q}\max\left\{u(y_j, t), u(z_j, t)\right\} \max\{{\lambda}^{{1\over q}-1} (1-\lambda), (1-\lambda)^{{1\over q}-1} \lambda\}(1-r^q)+{1\over j},
\end{aligned}
\end{equation}
where 
\begin{equation*}
r=\frac{\min\left\{u(y_j, t), u(z_j, t)\right\}}{\max\left\{u(y_j, t), u(z_j, t)\right\}}\in (0, 1].
\end{equation*}
%for some $r_0\in(0,1)$ due to \eqref{ineq:positive}, which is independent of $j$. 
Indeed, for any $0<a\leq b$, we have 
\begin{equation*}
b-(\lambda a^q+(1-\lambda) b^q)^{1\over q} \le b (g(1)-g(r^q)),
\end{equation*}
where $g(s)=(\lambda s +1-\lambda )^{1\over q}$ and $r=a/b\in (0, 1]$. Since $g$ is concave, we get
\begin{equation*}
g(1)-g(r^q)\leq g'(r^q)(1-r^q) ={1 \over q}\lambda (\lambda r^q+1-\lambda)^{{1\over q}-1}(1-r^q)\leq {1\over q}\lambda(1-\lambda)^{{1\over q}-1}(1-r^q).
\end{equation*}
This implies that 
\begin{equation*}
b-(\lambda a^q+(1-\lambda) b^q)^{1\over q}\leq {1\over q} b\lambda(1-\lambda)^{{1\over q}-1}(1-r^q).
\end{equation*}
If $a\geq b>0$, then we can similarly obtain 
\begin{equation*}
a-(\lambda a^q+(1-\lambda) b^q)^{1\over q}\leq {1\over q} a(1-\lambda)\lambda^{{1\over q}-1}(1-r^q)
\end{equation*}
with $r=b/a$. In general, for any $a, b>0$, we have 
\begin{equation*}
\max\{a, b\}-(\lambda a^q+(1-\lambda)b^q)^{1\over q}\leq {1\over q}\max\{a, b\}\max\{{\lambda}^{{1\over q}-1} (1-\lambda), (1-\lambda)^{{1\over q}-1} \lambda\}(1-r^q)
\end{equation*}
with $r=\min\{a, b\}/\max\{a, b\}$, which leads us to \eqref{attain-q app1}.

The relation \eqref{attain-q app1}, together with \eqref{eq:upper bdd}, immediately yields
\begin{equation*}
0\leq u_{\star, \lambda}(x, t)-u_{q, \lambda}(x, t)\leq {C\over q}\left(u_{\star, \lambda} (x, t)+1\right)+\frac{1}{j}
\end{equation*}
for some $C>0$ independent of $(x, t)$ and $j\geq 1$. Letting $j\to \infty$, we thus deduce the uniform convergence of $u_{q, \lambda}$ to $u_{\star, \lambda}$ in any compact subset of $\R^n\times[0,\infty)$ as $q\to \infty$. 
\end{proof}

\begin{rem}\label{rmk minimizer}
Under the coercivity condition on $u$, the infimum in \eqref{eq:q-envelop} can be obtained and the points $y_j, z_j$ in the proof above can be taken as minimizers. The convergence result certainly still holds.
\end{rem}

We show a key ingredient to prove \eqref{main-goal2}, which stems from the idea in \cite{ALL} to prove convexity of solutions to fully nonlinear equations by using its convex envelope. Such an idea is later developed in \cite{ILS} to show a power-type convexity or concavity with a finite exponent. We here makes a further step, studying the limit case as the exponent tends to $\infty$.

\begin{lem}\label{lem:super-general}%[Supersolution preserving property]
Assume that {\rm(F1)--(F7)} hold. Let $u\in C(\R^n\times [0, \infty))$ be a  supersolution of \eqref{E1} satisfying \eqref{ineq:positive} and \eqref{coercive}. Let $\lambda\in (0, 1)$ and $u_{\star, \lambda}$ be the function defined by \eqref{eq:envelop2-lam}. Then $u_{\star, \lambda}$ is  a supersolution of \eqref{E1}. 
\end{lem}

\begin{proof}
For simplicity of notation, we write 
$w_\star=u_{\star, \lambda}$ and $w_q=u_{q, \lambda}$.  Let us first show that $w_\star \in {\rm LSC}(\mathbb{R}^n \times [0, \infty))$. For an arbitrary $(x_0, t_0) \in \mathbb{R}^n \times [0,\infty)$, let $(x_j, t_j)$ be a sequence satisfying 
\[ (x_j, t_j) \to (x_0, t_0), \quad w_\star (x_j, t_j) \to \liminf_{(x, t) \to (x_0, t_0)} w_\star (x, t) \quad \text{as} \; \; j \to \infty. \]
Due to the coercivity \eqref{coercive}, there exist $y_j, z_j \in \mathbb{R}^n$ such that 
\begin{equation}\label{semi-continuity-w} 
x_j = \lambda y_j + (1-\lambda) z_j, \quad w_\star(x_j, t_j) = \min\{u(y_j, t_j), u(z_j, t_j)\}. 
\end{equation}
Since $x_j$ is a bounded sequence, if either of the sequences $y_j, z_j$ is unbounded, so does the other. Thus we can choose a subsequence such that $|y_j|, |z_j| \to \infty$, for which by  \eqref{coercive} again we have 
\[
u(y_j, t_j)\to \infty, \quad u(z_j, t_j)\to \infty\quad \text{as $j\to \infty$.}
\]
%Applying \eqref{coercive} again, it implies that 
It follows from \eqref{semi-continuity-w} that 
\[ w_\star(x_j, t_j)\to \infty \quad \text{as $j\to \infty$}, \]
which is a contradiction to the fact that $w_\star\leq u$ in $\R^n\times [0, \infty)$. 
Therefore, it is sufficient to assume that $y_j$ and $z_j$ are bounded sequences.  Choosing subsequences $y_j$ and $z_j$ converging to $y_0$ and $z_0$  in $\R^n$ respectively, we can take the limit of \eqref{semi-continuity-w} to obtain 
\[ \liminf_{(x, t) \to (x_0, t_0)} w_\star (x, t) = \min \{u(y_0, t_0), u(z_0, t_0)\} \ge w_\star (x_0, t_0). \] 
Hence, $w_\star \in {\rm LSC}(\mathbb{R}^n \times [0, \infty))$. 
The lower semicontinuity of $w_q$ can be proved similarly.

Let us next proceed to show that $w_\star$ satisfies the supersolution property. Suppose that there exist $(x_0, t_0)\in \R^n\times (0, \infty)$ and $\varphi\in C^2(\R^n\times (0, \infty))$ such that $w_\star-\varphi$ attains a strict minimum at $(x_0, t_0)$.
Without loss of generality we may assume $\varphi>0$ in $\R^n\times(0,\infty)$. 

 In light of Proposition \ref{prop:converge}, there exists a sequence, indexed by $q$, of $(x_q, t_q)\in \R^n\times (0, \infty)$ such that 
$w_q-\varphi$ attains a strict minimum at $(x_q, t_q)$ and 
\begin{equation*}
(x_q, t_q)\to (x_0, t_0), \quad w_q(x_q, t_q)\to w_\star(x_0, t_0)
\quad\text{as} \quad q\to \infty. 
\end{equation*} 
Due to \eqref{coercive}, there exist $y_q,\ z_q\in \R^n$ such that 
\begin{equation}\label{ineq:wq}
x_q=\lambda y_q+(1-\lambda) z_q, \quad
w_q(x_q, t_q)
=\left(\lambda u(y_q, t_q)^q+(1-\lambda) u(z_q, t_q)^q\right)^{\frac{1}{q}}. 
\end{equation}
Shifting $\varphi$ so that $\varphi(x_q, t_q)=w_q(x_q, t_q)$ and letting $v:=u^q$ and $\psi:=\varphi^q$, we see that 
\begin{equation*}
(y, z, t)\mapsto  \lambda v(y, t)+(1-\lambda) v(z, t)-\psi(\lambda y+(1-\lambda) z, t)
\end{equation*}
takes a minimum at $(y_q, z_q, t_q)\in\R^n\times\R^n\times(0,\infty)$. 
%Set 
%$\tilde{x}_j:=\lambda y_q+(1-\lambda) z_q$. 
%Then, 
%\[
%\tilde{x}_j\to x_0 \quad\text{and}\quad t_q\to t_0 
%\quad\text{as}\quad j\to\infty,  q\to\infty \ \text{in this order}. 
%\]

By the Crandall-Ishii lemma \cite{CIL}, for any $\varepsilon>0$ small we have $(h_q, \eta_q,  Y_q)\in \overline{P}^{2, -}v(y_q, t_q)$ and $(k_q, \zeta_q, Z_q)\in \overline{P}^{2, -} v(z_q, t_q)$ such that
\begin{equation}\label{ishii-t}
\lambda h_q+ (1-\lambda) k_q=\psi_t(x_q, t_q), 
\quad
\eta_q=\zeta_q=\nabla \psi(x_q, t_q),
\end{equation}
and
\begin{equation}\label{ishii-xd2}
\begin{aligned}
&\begin{pmatrix}
\lambda Y_q & 0\\
0 & (1-\lambda) Z_q
\end{pmatrix}\\
&\geq 
\begin{pmatrix}
\lambda^2 X_q & \lambda(1-\lambda)  X_q\\
\lambda(1-\lambda)  X_q & (1-\lambda)^2  X_q
\end{pmatrix}-\varepsilon\begin{pmatrix}
\lambda^2 X_q & \lambda(1-\lambda)  X_q\\
\lambda(1-\lambda)  X_q & (1-\lambda)^2  X_q
\end{pmatrix}^2, 
\end{aligned}
\end{equation}
where $X_q=\nabla^2 \psi(x_q, t_q)$. 
%By \eqref{ishii-t},  we have 
It follows that 
\[
\lambda \xi^T Y_q\xi+(1-\lambda) \xi^T Z_q\xi\geq \xi^T X_q\xi -C\varepsilon |\xi|^2 
\quad\text{for all} \ \xi\in\R^n, 
\]
where $C>0$ depends on the uniform bound of $|\nabla^2\psi(x_q, t_q)|$.  In other words, we get
\begin{equation}\label{ishii-hessian} 
\lambda Y_q+(1-\lambda)Z_q\geq X_q-C\varepsilon I.
\end{equation}

Since $u$ is a supersolution of \eqref{E1}, it is not difficult to show that $v$ is a supersolution of %due to Proposition \ref{prop:geometric}. 
\[
v_t+ G_\beta(v, \nabla v, \nabla^2 v, K\cap \{v(\cdot, t)\leq v(x, t)\})=0
\]
with $\beta=1-{1\over q}$, where $G_\beta$ is the transformed operator given by \eqref{transformed op}.
Adopting the definition of supersolutions at $(y_q, t_q)$ and $(z_q, t_q)$, we have
\begin{equation}\label{ineq:visc}
\begin{aligned}
&h_q+G_\beta(a_q, \eta_q, Y_q, K\cap \{v(\cdot, t_q)\leq v(y_q, t_q)\})\geq 0, \\
&k_q+G_\beta(b_q, \zeta_q, Z_q, K\cap \{v(\cdot, t_q)\leq v(z_q, t_q)\})\geq 0,
\end{aligned}
\end{equation}
where $a_q=v(y_q, t_q), b_q=v(z_q, t_q)$.

Applying Proposition \ref{prop:equiv}, we can divide our argument into the following two cases: 
%\begin{align*}
\begin{itemize}
\item Case 1. $\nabla \varphi(x_0, t_0)\neq 0$,
\item Case 2. $\nabla \varphi(x_0, t_0)= 0, \ \text{and} \ \nabla^2\varphi(x_0,t_0)=0$. 
\end{itemize}
We write $\xi_q=\nabla\psi(x_q, t_q)$ for simplicity of notation.  

In Case 1, by \eqref{ishii-t}, we have $\nabla\psi(x_q, t_q)=\eta_q=\zeta_q\neq 0$ when $q>1$ is sufficiently large. 
%We thus get
%\begin{align*}
%&h_q+F(\eta_q, Y_q,K\cap \{v(\cdot, t_q)\leq v(y_q, t_q)\})\geq 0, \\ 
%&k_q+F(\zeta_q, Z_q, K\cap\{v(\cdot, t_q)\leq v(z_q, t_q)\})\geq 0.
%\end{align*}
Multiplying the first inequality in \eqref{ineq:visc} by $\lambda$ and the second by $1-\lambda$ and then adding them up, by \eqref{ishii-t} we are led to
\[
\begin{aligned}
&\psi_t(x_q, t_q)+\lambda G_\beta(a_q, \xi_q, Y_q, W_\star[x_0, t_0]\})+(1-\lambda) G_\beta(b_q, \xi_q, Z_q, W_\star[x_0, t_0]))\\
&\quad \geq 
\lambda \left(G_\beta(a_q, \xi_q, Y_q, W_\star[x_0, t_0])-G_\beta(a_q, \xi_q, Y_q, U[y_q, t_q])\right)\\
&\qquad+(1-\lambda) \left(G_\beta(b_q, \xi_q, Z_q, W_\star[x_0, t_0])-G_\beta(b_q, \xi_q, Z_q, U[z_q, t_q])\right),
\end{aligned}
\]
where we denote, for any $(x, t)\in \R^n\times (0, \infty)$,
\begin{align*}
&W_\star[x, t]:=K\cap\{w_\star(\cdot, t)\leq w_\star(x, t)\}, \\
&U[x, t]:=K\cap\{u(\cdot, t)\leq u(x, t)\}=K\cap\{v(\cdot, t)\leq v(x, t)\}.
\end{align*}
Applying (F1), (F7) and \eqref{ishii-hessian} and noticing that 
\[
\lambda a_q+(1-\lambda) b_q= w_q(x_q, t_q)^q, 
\]
 we then get
\begin{equation*}
\begin{aligned}
&\psi_t(x_q, t_q)+G_\beta\left(w_q(x_q, t_q)^q, \nabla\psi(x_q, t_q), \nabla^2\psi(x_q, t_q)-C\varepsilon I, W_\star[x_0, t_0]\right)
\\
&\geq \lambda \left(G_\beta(a_q, \xi_q, Y_q, W_\star[x_0, t_0])-G_\beta(a_q, \xi_q, Y_q, U[y_q, t_q])\right)\\
&\qquad+(1-\lambda) \left(G_\beta(b_q, \xi_q, Z_q, W_\star[x_0, t_0])-G_\beta(b_q, \xi_q, Z_q, U[z_q, t_q])\right).
\end{aligned}
\end{equation*}
%By direct calculations, we have 
%\[
%\psi_t(x_q, t_q)=q\varphi(x_q, t_q)^{q-1}\varphi_t(x_q, t_q), \quad
%\nabla\psi(x_q, t_q)=q\varphi(x_q, t_q)^{q-1}\nabla\varphi(x_q, t_q),
%\]
Rewriting this relation in terms of the operator $F$, we are led to 
%\[
%\psi_t(x_q, t_q)+G_\beta\left(w_q(x_q, t_q)^q, \nabla\psi(x_q, t_q), \nabla^2\psi(x_q, t_q)-C\varepsilon I, W_\star[x_0, t_0]\right)\geq \lambda D_{1, q}+(1-\lambda) D_{2, q}.
%\]
%By (F1), we  obtain
\begin{equation}\label{eq:vis}
\begin{aligned}
&\varphi_t(x_q, t_q)+F(w_q(x_q, t_q), \nabla\varphi(x_q, t_q), \nabla^2\varphi(x_q, t_q)-C_q\varepsilon I, W_\star[x_0, t_0])\\
&\quad \geq \lambda \frac{u(y_q, t_q)^{q-1}}{\varphi(x_q, t_q)^{q-1}} D_{1, q}+ (1-\lambda) \frac{u(z_q, t_q)^{q-1}}{\varphi(x_q, t_q)^{q-1}}D_{2, q},
\end{aligned}
\end{equation}
where 
\[
C_q={C\over q}\varphi(x_q, t_q)^{1-q}, 
\] 
\begin{equation*}
\begin{aligned}
&D_{1, q}=F(u(y_q, t_q), \nabla \varphi(x_q, t_q), Y_q', W_\star[x_0, t_0])-F(u(y_q, t_q), \nabla \varphi(x_q, t_q), Y_q', U[y_q, t_q]),\\
&D_{2, q}=F(u(z_q, t_q), \nabla \varphi(x_q, t_q), Z_q', W_\star[x_0, t_0])-F(u(z_q, t_q), \nabla \varphi(x_q, t_q), Z_q', U[z_q, t_q])
\end{aligned}
\end{equation*}
with
\begin{equation*}
\begin{aligned}
&Y_q'={1\over q}\varphi(x_q, t_q)^{1-q}Y_q-{q-1\over q^2} \varphi(x_q, t_q)^{1-2q} \nabla\varphi(x_q, t_q)\otimes \nabla \varphi(x_q, t_q), \\
&Z_q'={1\over q}\varphi(x_q, t_q)^{1-q} Z_q-{q-1\over q^2} \varphi(x_q, t_q)^{1-2q} \nabla\varphi(x_q, t_q)\otimes \nabla \varphi(x_q, t_q). 
\end{aligned}
\end{equation*}

The assumption that $u\geq c_0$ yields 
\[
\begin{aligned}
&\frac{\lambda u(y_q, t_q)^{q-1}+(1-\lambda) u(z_q, t_q)^{q-1}}{\varphi(x_q, t_q)^{q-1}}=\frac{\lambda u(y_q, t_q)^{q-1}+(1-\lambda) u(z_q, t_q)^{q-1}}{w_q(x_q, t_q)^{q-1}}\\
&\leq \frac{\lambda u(y_q, t_q)^{q}+(1-\lambda) u(z_q, t_q)^{q}}{c_0w_q(x_q, t_q)^{q-1}}= \frac{w_q(x_q, t_q)^q}{c_0w_q(x_q, t_q)^{q-1}}=\frac{w_q(x_q, t_q)}{c_0},
\end{aligned}
\]
which implies 
\begin{equation}\label{coefficient bound}
\sup_{q>1}\left\{\frac{\lambda u(y_q, t_q)^{q-1}+(1-\lambda) u(z_q, t_q)^{q-1}}{\varphi(x_q, t_q)^{q-1}}\right\}<\infty.
\end{equation}
%Notice that 
%\[ \lambda v(y_q,  t_q) + (1-\lambda) v(z_q,  t_q) - \psi(x_q,  t_q) + \psi(x_q, t_q) \le \lambda u(y_q, t_q)^q + (1-\lambda) u(z_q, t_q)^q, \text{\fbox{\bl{is not necessary?}}}\]
Let us proceed to estimate $D_{1, q}, D_{2, q}$ in \eqref{eq:vis}. Note that by \eqref{ineq:wq}  
\begin{equation*}
%\limsup_{q\to \infty} u(y_q,  t_q)
%=
\limsup_{q\to \infty} u(y_q, t_q)
\leq \limsup_{q\to \infty} w_q(x_q, t_q)=w_\star(x_0, t_0).
\end{equation*}
Also, it is easily seen that
\[
w_\star(\cdot, t_0)=u_{\star, \lambda}(\cdot, t_0) 
\leq 
u(\cdot,t_0)
\le \liminfs_{q\to \infty} u(\cdot,  t_q)\quad \text{in $\R^n$}.
\] 
Thus, in light of Lemma \ref{lem meas} we have %, for any $\varepsilon>0$,  
\begin{equation*}
m(U[y_q,  t_q]\setminus W_\star[x_0, t_0])\to 0\quad \text{as $q\to \infty$.}
\end{equation*} 

Since (F4) implies 
\[
\begin{aligned}
&F(u(y_q, t_q), \nabla \varphi(x_q,  t_q), Y_q', W_\star[x_0, t_0])\\
&\geq F(u(y_q, t_q), \nabla \varphi(x_q,  t_q), Y_q', W_\star[x_0, t_0]\cap U[y_q,  t_q]),
\end{aligned}
\]
%This implies that $\{u(\cdot, t_q)\leq u(y_q, t_q)\}$ are bounded uniformly for $q>1$ large and 
%\begin{equation}
%U^\#_y:=\bigcap_{k>1}\bigcup_{q\geq k}\{u(\cdot, t_q)\leq u(y_q, t_q)\}\subset \{w_\star(\cdot, t_0)\leq w_\star(x_0, t_0)\}. 
%\end{equation}
in view of \eqref{usc-set} in (F3), we deduce that 
\[%\label{eq:meas1}
D_{1, q}\geq -\omega_R\bigg(m\big(U[y_q,  t_q]\setminus W_\star[x_0, t_0]\big)\bigg)
\]
for $q>1$ sufficiently large, where $R=|\nabla \varphi(x_0, t_0)|+1$.
Similarly, we have 
\[%\label{eq:meas2}
%D_{2, q}\geq -\omega_R\left({1\over q}\right).
%D_{2, q}\geq -\omega_R\bigg(\big(m(\{u(\cdot, t_q)\leq u(z_q, t_q)\})-m(U^\#_z)\big)_+\bigg),
D_{2, q}\geq -\omega_R\bigg(m\big(U[z_q,  t_q]\setminus W_\star[x_0, t_0]\big)\bigg)
\]
with
\begin{equation*}
m(U[z_q,  t_q]\setminus W_\star[x_0, t_0])\to 0\quad \text{as $q\to \infty$.}
\end{equation*}
Hence, thanks to \eqref{coefficient bound}, sending $\varepsilon \to 0$ and then $q\to \infty$ in \eqref{eq:vis}, 
we get 
\[%\label{ineq:sup}
\varphi_t(x_0, t_0)+F(w_\star(x_0, t_0), \nabla \varphi(x_0, t_0), \nabla^2 \varphi(x_0, t_0), K\cap\{u_{\star, \lambda}(\cdot, t_0)\leq u_{\star, \lambda}(x_0, t_0)\})\geq 0.
\]
%\begin{equation}\label{eq:perturb}
%\varphi_t(x_q, t_q)+F(\nabla\varphi(x_q, t_q), \nabla^2\varphi(x_q, t_q)-c\varepsilon I, \{w_\star(\cdot, t_0)\leq w_\star(x_0, t_0)\})\geq -\omega\left({1\over q}\right)
%\end{equation}
%for some modulus $\omega$. 

The proof in Case 1 is complete. Let us next discuss Case 2. If $\nabla \varphi(x_q,  t_q)\neq 0$ along a subsequence, then passing to the limit of  \eqref{eq:vis} as $\varepsilon\to 0$ and $q\to \infty$ via the subsequence, 
by (F6) we get 
\begin{equation}\label{ineq:sup2}
\varphi_t(x_0, t_0)+\mu(w_\star(x_0, t_0))\geq 0,
\end{equation}
as desired.
 %$\varphi_t(x_0, t_0)\geq 0$.  
%which shows that $w_\star$ we can still obtain \eqref{ineq:sup}.  

It remains to consider the case when $\nabla \varphi(x_q,  t_q)= 0$ for all $q>1$ large. By \eqref{ishii-xd2}, we deduce that 
\begin{equation}\label{van-grad1}
Y_q\geq \lambda(1-\varepsilon) X_q, \quad Z_q\geq (1-\lambda)(1-\varepsilon)X_q,
\end{equation}
where $X_q=\nabla^2\varphi(x_q,  t_q)\to 0$ as $q\to \infty$.  As in Case 1, we take $\beta=1-{1\over q}$ with $q>0$ large.
Adopting the definition of supersolutions in this case, we have 
\[
\begin{aligned}
& %{1\over q}\varphi(x_q, t_q)^{1-q}  
h_q+G_\beta^\ast(v(y_q, t_q), 0, Y_q, K \cap\{v(\cdot,  t_q)\leq v(y_q,  t_q)\})\geq 0, \\
& %{1\over q}\varphi(x_q, t_q)^{1-q} 
k_q+G_\beta^\ast(v(z_q, t_q), 0, Z_q, K \cap\{v(\cdot,  t_q)\leq v(z_q,  t_q)\})\geq 0,
\end{aligned}
\]
which by (F4) yields
\[
\begin{aligned}
&h_q+G_\beta^\ast(v(y_q, t_q), 0, Y_q, K)\geq 0, \\
& k_q+G_\beta^\ast(v(z_q, t_q), 0, Z_q, K)\geq 0.
\end{aligned}
\]
In view of \eqref{van-grad1} and (F6), we have 
\[
\begin{aligned}
&(1-\beta)h_q+v(y_q, t_q)^{\beta}\mu(v(y_q, t_q)^{1-\beta})\geq -v(y_q, t_q)^{\beta}\omega(|X_q|),\\
&(1-\beta)k_q+v(z_q, t_q)^{\beta}\mu(v(z_q, t_q)^{1-\beta})\geq -v(z_q, t_q)^{\beta}\omega(|X_q|).
\end{aligned}
\]
It follows from (F7) that 
\[
\begin{aligned}
&{1\over q}(\lambda h_q+(1-\lambda) k_q)+w_q(x_q, t_q)^{q-1} \mu(w_q(x_q, t_q))\\
&\geq -\lambda v(y_q, t_q)^{1-{1\over q}}\omega(|X_q|)-(1-\lambda)v(z_q, t_q)^{1-{1\over q}}\omega(|X_q|),
\end{aligned}
\]
which, together with the relation $w_q(x_q, t_q)=\varphi(x_q, t_q)$, yields, 
\[
\begin{aligned}
&{1\over q}\varphi(x_q, t_q)^{1-q}(\lambda h_q+(1-\lambda) k_q)+\mu(w_q(x_q, t_q))\\
&\geq -w_q(x_q, t_q)^{1-q}\left (\lambda u(y_q, t_q)^{q-1}+ (1-\lambda)u(z_q, t_q)^{q-1}\right) \omega(|X_q|).
\end{aligned}
\]
Noticing that, due to \eqref{ishii-t}, 
\[
\varphi_t(x_q,  t_q)={1\over q} \varphi(x_q,  t_q)^{1-q} \psi_t(x_q,  t_q)={1\over q} \varphi(x_q,  t_q)^{1-q} (\lambda h_q+(1-\lambda)k_q), 
\]
we obtain 
\begin{equation}\label{van-grad2}
\varphi_t(x_q, t_q)+\mu(w_q(x_q, t_q))\geq - \frac{\lambda u(y_q, t_q)^{q-1}+(1-\lambda) u(z_q, t_q)^{q-1}}{w_q(x_q, t_q)^{q-1}}\omega(|X_q|).
\end{equation}
In view of \eqref{coefficient bound}, letting $q\to \infty$ in \eqref{van-grad2}, we again end up with \eqref{ineq:sup2}.
%\[
%\varphi_t(x_0, t_0)+\mu(w_\star(x_0, t_0))\geq 0.
%\]
\end{proof}

%where, due to the condition that $\nabla \varphi(x_q,  t_q)= 0$, $Y_q'$ and $Z_q'$ reduce to 
%\[
%Y_q'={1\over q}\varphi(x_q, t_q)^{1-q}Y_q, \quad Z_q'={1\over q}\varphi(x_q, t_q)^{1-q}Z_q.
%\]

%Using \eqref{van-grad1}, (F1) and (F4), we get
%\[
%\begin{aligned}
%0 & \le {1\over q}\varphi(x_q, t_q)^{1-q} h_q+F^\ast\left(u(y_q, t_q), 0,  {\lambda\over q} (1-\varepsilon)\varphi(x_q, t_q)^{1-q} X_q, K\cap\{u(\cdot,  t_q)\leq u(y_q,  t_q)\}\right)\\
% \leq{1\over q}\varphi(x_q, t_q)^{1-q} h_q+F^\ast\left(u(y_q, t_q), 0,  {\lambda\over q} (1-\varepsilon)\varphi(x_q, t_q)^{1-q} X_q, K\right)
%\end{aligned}
%\]
%and similarly 
%\[
%0 \leq{1\over q}\varphi(x_q, t_q)^{1-q} k_q+F^\ast\left(u(z_q, t_q), 0,  {\lambda\over q} (1-\varepsilon)\varphi(x_q, t_q)^{1-q} X_q, K\right)
%\end{equation*}
%Sending $\varepsilon\to 0$ and $q\to \infty$, by (F1) and (F3) we have, 
%\begin{equation}
%\limsup_{q\to \infty}{1\over q}\varphi(x_q, t_q)^{1-q} (\lambda h_q+(1-\lambda) k_q)+\lambda \mu(u(y_q, t_q))+(1-\lambda)\mu(u(z_q, t_q)) \geq 0.
%\end{equation}
%Letting $j, q\to \infty$ yields
%$\varphi_t(x_0, t_0) \geq 0$. 
%Theorem \ref{thm:key}

\begin{proof}[Proof of Theorem \ref{thm:main}]
By the quasiconvexity of $u_0$, we have $u_{\star, \lambda}(\cdot, 0)=u_0$ in $\R^n$.
Moreover, since $c_0\leq u_{\star, \lambda}\leq u$ holds in $\R^n\times [0, \infty)$, $u_{\star, \lambda}$ obviously satisfies the growth condition \eqref{growth comp}.
The relation \eqref{main-goal2} is then an immediate consequence of Lemma \ref{lem:super-general} and the comparison principle, Theorem \ref{thm:comp}. Noticing that \eqref{main-goal2} implies the quasiconvexity of $u$ in space, we complete the proof of Theorem \ref{thm:main}, 
\end{proof}

\begin{rem}\label{rem lip}
%Suppose that $u$ is known to be uniformly positive, i.e., $u\geq c_0$ in $\R^n\times [0, \infty)$ for some $c_0>0$. Our proof of Lemma \ref{lem:super-general} still works as long as the concavity condition of $(r, X)\mapsto G_\beta(r, p, X, A)$ holds with restriction that $r\geq c_0$. 
We can weaken the assumption (F7) if the solution $u$ satisfies additional regularity assumptions.
If $x\mapsto u(x, t)$ is assumed to be $L$-Lipschitz uniformly for all $t\geq 0$, we can show Lemma \ref{lem:super-general} and thus Theorem \ref{thm:main} under the following relaxed version of \eqref{concave1} in (F7): for any $\beta<1$ close to $1$,
%$(r, X)\mapsto G_\beta(r, p, X, A)$ is concave in $(0, \infty)\times \bS^n$ 
%Suppose that and $x\mapsto u(x, t)$ is $L$-Lipschitz for all $t\geq 0$. 
%Our proof of Lemma \ref{lem:super-general} still works as long as the concavity condition of $(r, X)\mapsto G_\beta(r, p, X, A)$ holds 
%with restriction that $r\geq c_0$ and $0<|p|\leq L$. That is to say, we can replace \eqref{concave1} in (H7) by the following: 
$(r, X)\mapsto G_\beta(r, p, X, A)$ is concave in $[c_0, \infty)\times \bS^n$ for any $p\in \R^n\setminus\{0\}$ with $|p|\leq L$ and $A\in \cB_K$. 
\end{rem}

\begin{comment}
\begin{rem}
One may extend the quasiconvexity preserving result to a more general class of nonlinear parabolic equations including
\begin{equation}
u_t+F(u, \nabla u, \nabla^2u, \{u(\cdot, t)<u(x, t)\})=f(x, t)\quad \text{in $\R^n\times (0, \infty)$,}
\end{equation}
where $F$ satisfies (F1)--(F6) and $f$ is a given bounded uniformly continuous function in $\R^n\times [0, \infty)$. In this case, the concavity assumption (F7) is generalized in the following manner. For any $\beta<1$ close $1$, we demand: (i)  for any $t\in (0, \infty)$, $p\in \R^n\setminus \{0\}$ and $A\in \cB_K$, 
\[
(x, r, X)\mapsto G_\beta(r, p, X, A)-{1\over 1-\beta} r^\beta f(x, t)
\]
is concave in $\R^n\times [c_0, \infty)\times \bS^n$; (ii) for any $t\in (0, \infty)$, 
\[
(x, r)\mapsto r^\beta \mu(r^{1-\beta})-r^\beta f(x, t)
\]
is concave in $\R^n\times [c_0, \infty)$. Due to the boundedness of $f$, these conditions require that $x\mapsto f(x, t)$ be constant. 
{\color{blue}
source項の$f$が結局$x$に対して定数しか扱えないということでしたら
無理にこのRemarkを入れる必要ないかなとも思ってきました．
$f$を非有界と拡張するの可能性は依然として考えられるとは思いますが，
多くの部分で修正が必要となってくるので，無理をするのをやめましょう．
}
\end{rem}
\end{comment}

Although our method relies on the power convex envelope $w_q$, in general one cannot expect the preservation of spatial power convexity to hold for any finite exponent $q$, as indicated by the simple example below. 
\begin{ex}\label{ex:1}
Consider the first order equation in two space dimensions:
\begin{equation}\label{eq:ex2}
u_t+|\nabla u| m(K\cap \{u(\cdot, t)<u(x, t)\})=0 \quad \text{in $\R^2\times (0, \infty)$}, 
\end{equation}
where $K\subset \R^2$ is taken to be the closed ball $\ol{B_R(0)}$ centered at $0$ with radius $R>0$.
Let the initial value be
\begin{equation}\label{eq:ex2initial}
u_0(x)=|x|+1,\quad \text{ $x\in \R^2$.}
\end{equation}
Since $u_0$ is radially symmetric and attains a minimum at $x=0$, we can consider radially symmetric solutions $u(x,t)=\varphi(|x|,t)$ 
(see \cite{GMT}), where $\varphi$ satisfies 
\begin{equation}\label{eq:ex1}
\left\{
\begin{array}{ll}
\varphi_t+|\varphi_r|\pi\min\{R^2,r^2\}=0 & \text{for} \ r>0, t>0, \\
\varphi(0, t)=1 & \text{for}\ t> 0,\\
\varphi(r,0)=r+1 & \text{for} \ r\ge0. 
\end{array}
\right. 
\end{equation}
Using the optimal control formula (see \cite{BC, Hung-book} for instance), we can express the viscosity solution to \eqref{eq:ex1} as %is represented by \fbox{\ma{$\phi(\gamma(s)) \mapsto \phi(\gamma(s),0)$, $|\dot\gamma(s)|^2 \mapsto |\dot\gamma(s)|$}}
\[
\varphi(r,t)=
\inf\big\{\varphi(\gamma(t),0)\mid 
\gamma(0)=r, 
|\dot{\gamma}(s)|\le \pi\min\{R^2, \gamma(s)^2\} \ \text{for all} \ s\in[0,t]
\big\}.  
\]
The optimal control $\gamma^\ast$ in this setting is rather simple. Indeed, for any given $r\geq 0$ and $t\geq 0$, 
if $r\le R$ then $\gamma^\ast$ solves 
\[
\left\{
\begin{array}{ll}
\dot{\gamma}^\ast(s)=-\pi\gamma^\ast(s)^2 & \text{for} \ s>0, \\
\gamma^\ast(0)=r, &  
\end{array}
\right.
\]
and thus $\gamma^\ast(t)=r/(1+\pi r t)$. 
If $r>R+\pi R^2 t$, then $\gamma^\ast$ is the solution to 
\[
\left\{
\begin{array}{ll}
\dot{\gamma}^\ast(s)=-\pi R^2 & \text{for} \ s>0, \\
\gamma^\ast(0)=r, &  
\end{array}
\right. 
\]
which is $\gamma^\ast(t)=r-\pi R^2 t$. 
Finally, if $R<r<R+\pi R^2 t$, then letting $\tilde{t}:=(r-R)/(\pi R^2)$,
we see that $\gamma^\ast$ satisfies
\[
\left\{
\begin{array}{ll}
\dot{\gamma}^\ast(s)=-\pi R^2 & \text{for} \ 0<s<\tilde{t}\\
\gamma^\ast(0)=r, &  
\end{array}
\right. 
\quad\text{and}\quad
\left\{
\begin{array}{ll}
\dot{\gamma}^\ast(s)=-\pi\gamma^\ast(s)^2 & \text{for} \ \tilde{t}<s<t, \\
\gamma^\ast(\tilde{t})=R. &  
\end{array}
\right. 
\]
By elementary computations, we obtain %$\gamma^\ast(s)=\frac{R}{1+\pi R(s-\tilde{t})}$ for $s \ge \tilde{t}$, which implies 
\[
\gamma^\ast(t)=\frac{R^2}{\pi R^2 t-r+2R}. 
\]
Hence, the unique viscosity solution of \eqref{eq:ex2} can be written as
\begin{equation}\label{eq:ex2sol}
u(x, t)=\begin{cases}
\frac{|x|}{1+\pi |x| t}+1 & \text{if $|x|\leq R$,}\\
\frac{R^2}{\pi R^2t-{|x|}+2R}+1 & \text{if $R<|x|\leq R+\pi R^2 t$,}\\
|x|-\pi R^2 t+1 & \text{if $|x|> R+\pi R^2 t$.}
\end{cases}
\end{equation}
It is not difficult to see that $x\mapsto u(x, t)$ is not convex (for $|x|<R$) for any $t>0$ in spite of the convexity of $u_0$. 

Moreover, since the operator $F$ in this case is geometric, satisfying $F(cp, A)=cF(p, A)$ for all $c\in \R$, we can follow \cite[Theorem 4.2.1]{GBook} to prove that $g\circ u$ is still a solution of \eqref{eq:ex2} for any nondecreasing continuous function $g: \R\to \R$. 
In particular, we deduce that $u^{1\over q}$ (with $u$ given by \eqref{eq:ex2sol}) is the unique solution of \eqref{eq:ex2} with initial value $u_0^{1\over q}$ for any $q\geq 1$. This means that in general solutions of \eqref{eq:ex2} may fail to preserve $q$-convexity for every $q\geq 1$.
\end{ex}

In the example above, we obtained the solution of \eqref{eq:ex2} and \eqref{eq:ex2initial} by means of a control-based interpretation. It is not straightforward at all to find a representation formula for solutions of more general problems. If we slightly change the equation and consider
\[
u_t+\big\{m(K\cap \{u(\cdot, t)<u(x, t)\})-1\big\}|\nabla u| =0 \quad  \text{in $\R^n\times (0, \infty)$,}
\]
then the Hamiltonian, depending in a nonlocal manner on $u$, is neither coercive nor convex in general. The optimal control interpretation in this case becomes rather complicated. We study this problem in our upcoming work \cite{KLM2}.

%%%%%%%%%%%%%%%%%%%%%%%%%
\begin{comment}
{\color{red}
Here, we would like to point out an interesting point from the view point of the optimal control theory. 
\begin{ex}
Consider 
\[
\left\{
\begin{array}{ll}
u_t+\big\{m(K\cap \{u(\cdot, t)<u(x, t)\})-1\big\}|\nabla u| =0 & \text{in $\R^n\times (0, \infty)$}\\
u(x,0)=\varphi_0(|x|) & \text{in}  \ \R^n, 
\end{array}
\right. 
\]
where $K=\overline{B_R(0)}$, $\varphi_0\in\UC([0,\infty))$ is positive, strictly increasing and satisfy 
$\varphi_0(r)\to\infty$ as $r\to\infty$. By a similar argument to that in Example \ref{ex:1}, we have the 
unique viscosity solution $u(x,t)=\varphi(|x|,t)$, where 
$\varphi$ is the solution to 
\begin{equation}\label{eq:ex2}
\left\{
\begin{array}{ll}
\varphi_t+(\omega_n\min\{R,r\}^n-1)|\varphi_r|=0 & \text{for} \ r>0, t>0, \\%\text{\fbox{\ma{$\pi\min\{R^2,r^2\}$ is removed}}}\\
\varphi(r,0)=\varphi_0(r) & \text{for} \ r\ge0, 
\end{array}
\right. 
\end{equation}
where we denote by $\omega_n$ the volume of $n$-dimensional unit sphere. %that is, $\omega_n=\pi^{\frac{n}{2}}/\Gamma(\frac{n}{2}+1)$. 
%Here, $\Gamma$ is the Gamma function. 

Set $H(r,p):=(\omega_n\min\{R,r\}^n-1)|p|$. Letting $r^\ast$ be the constant satisfying 
\[
\omega_n\min\{R,r^\ast\}^n=1,
\] 
we can easily see that 
$p\mapsto H(r,p)$ is concave for $0<r<r^\ast$, and 
convex for $r^\ast<r$. 
Therefore, we can observe that the optimal control interpretation is rather complicated. We 
study this in the upcoming work in \cite{KLM2}. 

\end{ex}
}
\end{comment}
%%%%%%%%%%%%%%%%%

\section{Applications}\label{sec:applications}

To conclude our paper, let us provide several applications of our quasiconvexity preserving result. The parabolic operators in the examples below satisfy all of the assumptions in Theorem \ref{thm:main} including the assumption (F7). We also verify the assumption (I) for our existence result, Theorem \ref{thm:exist}.

\subsection{Nonlocal geometric flows}\label{sec:nonlocal-geo}
Let us study the anisotropic surface evolution given by \eqref{eq:ex1}. 
For this type of evolution, $\xi$ is the Cahn-Hoffman vector associated to an (extended) surface energy density $\gamma\in C(\R^n)\cap C^2(\R^n\setminus \{0\})$, which is  positively homogeneous of degree $1$ in $\R^n$, and satisfies $\gamma>0$ and $\nabla \gamma=\xi$ in $\R^n\setminus \{0\}$.
In this case, the level set formulation leads us to the following PDE:
\begin{equation}\label{nonlocal-geometric}
u_t+a|\nabla u|+b |\nabla u| m(K\cap \{u(\cdot, t)<u(x, t)\}) -c|\nabla u|\tr\left(\nabla^2 \gamma\left(-{\nabla u}\right)\nabla^2 u\right)=0.
\end{equation}
in $\R^n\times (0, \infty)$, where $a\in\R$, $b\ge 0$ and $c\ge0$ are given constants.  
The operator $F$ can be expressed as 
\[
F(p, X, A)=a|p|
+b|p| m(A)-|p|\tr\left(\nabla^2 \gamma(-p)X\right).
\]
We can show that $F$ satisfies the condition (F7). In fact, we can verify this property for a more general class of geometric operators $F$ and thus recover the results on convexity preserving property in \cite{C}. Suppose that \eqref{geometrical-pro} holds.
 It is clear that the transformed operator 
\[
G_\beta(r, p, X, A)=F(p, X, A)
\] 
for all $0<\beta<1$. Then \eqref{concave1} holds provided that $X\mapsto F(p, X, A)$ is concave, i.e., 
\[
\lambda F(p, X, A)+(1-\lambda)F(p, Y, A)\leq F(p, \lambda X+(1-\lambda)Y, A)
\]
for all $\lambda\in (0, 1)$, $p\in \R^n\setminus \{0\}$, $X, Y\in\bS^n$ and $A\in \cB_K$.
In addition, since $\mu=0$ in this case, we have \eqref{concave2} as well. Hence, (F7) is verified.

Furthermore, for the anisotropic surface evolution equation \eqref{nonlocal-geometric} with $\gamma$ convex, we can construct a  function $\phi$ satisfying (I) provided that $u_0$ is uniformly continuous, positive and coercive in $\R^n$. % satisfying 
%\begin{equation}\label{initial-divergence} 
%\inf_{|x| \ge R} u_0(x) \to \infty \quad \text{as} \; \; R \to \infty 
%\end{equation}
%as follows. 
Since \eqref{nonlocal-geometric} is geometric, %invariant with respect to the vertically translations of $u$, 
without loss of generality we may assume that $u_0 \ge c_0$ for some $c_0>0$. 
We can choose an increasing and coercive function $\sigma \in C^\infty([0, \infty)) \cap \UC([0,\infty))$ such that 
\begin{equation}\label{setting-sigma}
\inf_{\gamma(x) \ge R} u_0(x) \ge \sigma(R) \ge c_0 \quad \text{for any} \; \; R \in [0, \infty). %\quad \lim_{R \to \infty} \sigma(R) = \infty.
\end{equation}
%A function $\phi$ satisfying the condition (I) can be constructed as 
Take
\begin{equation}\label{setting-phi} 
\phi(x,t) := % \begin{cases}
\sigma\left(\max\{\gamma(x) - Ct, 0\}\right), %  \text{for} \; \; \gamma(x) \ge Ct, \\
%\sigma(0) & \text{for} \; \; \gamma(x) < Ct, 
%\end{cases} 
\end{equation}
where 
%\[ C := \left(\sup_{p \in \mathbb{R}^n \setminus \{0\}}|\nabla \gamma(x)|\right) \left\{|a| \left(\sup_{p \in \mathbb{R}^n,\ |p|=1} |M(p)|\right) + b m(K)\right\}. \]
\[ C := \left(\sup_{p \in \mathbb{R}^n \setminus \{0\}}|\nabla \gamma(x)|\right) \left\{|a| + b m(K)\right\}. \]
Note that the supremum of $|\nabla \gamma(x)|$ can be attained since $\nabla \gamma(x)$ is positively homogeneous of degree $0$. 
Since any level set $\{x \in \mathbb{R}^n: \gamma(x) = h\}$ is a scaled Wulff shape and its anisotropic mean curvature $\tr \left(\nabla^2 \gamma(-\nabla \gamma(x)) \nabla^2 \gamma(x)\right)\geq 0$
%is non-negative 
if $\gamma$ is convex,  by \eqref{geometrical-pro} we have % applying the geometrical property as in 
\[ \tr(\nabla^2 \gamma(-\nabla \phi(x,t)) \nabla^2 \phi(x,t)) = \sigma'(\gamma(x) - Ct) \tr \left(\nabla^2 \gamma(-\nabla \gamma(x)) \nabla^2 \gamma(x)\right) \ge 0 \]
if $\gamma(x) > Ct$. 
Therefore, we have 
\begin{align*}
&\;\phi_t(x,t) + F(\nabla \phi(x,t), \nabla^2 \phi(x,t), K \cap \{\phi(\cdot,t) < \phi(x,t)\}) \\
\le &\; \sigma'(\gamma(x) - Ct) \left\{ -C + |\nabla \gamma(x)| \Big(- a 
 + b m(K)\Big)\right\} \le 0  
\end{align*}
which ensures (v) in (I). 
The other conditions in (I) hold obviously due to the choice of $\sigma$ and the definition of $\phi$.
%\fbox{\eqref{hj eq}も含めて$u_0 \ge c_0$の仮定は本質的ではないことは触れなくて良い？}

\subsection{Nonlocal evolution equations depending on $u$}

We now consider a slightly different type of nonlocal curvature flows,  which resembles those local ones arising in the study of crystal growth  \cite{TG}. A typical equation reads
\begin{equation}\label{nonlocal-evolution}
u_t+V(u)+|\nabla u| m(K\cap \{u(\cdot, t)<u(x, t)\}) -\tr\left(\left(I-{\nabla u\otimes \nabla u\over |\nabla u|^2}\right)\nabla^2 u\right)=0,
\end{equation}
in $\R^n\times (0, \infty)$, where $V\in C^2(\R)$ is a given function.  Then the operator $F$ is 
\[
F(r, p, X, A)=V(r)+|p| m(A)-\tr\left(\left(I-{p\otimes p\over |p|^2}\right)X\right).
\]
We see that $F$ satisfies (F7) under appropriate assumptions on $V$. Indeed, for $p\neq 0$, the operator $G_\beta$ as in  \eqref{transformed op} is 
\[
G_\beta(r, p, X, A)={1\over 1-\beta} r^{\beta} V(r^{1-\beta})+|p| m(A)-\tr\left(\left(I-{p\otimes p\over |p|^2}\right)X\right).
\]
It is clear that, for $0<\beta<1$ close to $1$, the condition \eqref{concave1} holds in this case if $r\mapsto r^\beta V(r^{1-\beta})$ is concave in $(0, \infty)$. By direct computations, we get
\[
(r^\beta V(r^{1-\beta}))''=\beta(1-\beta)r^{-1} \left(V'(r^{1-\beta}) -r^{\beta-1} V(r^{1-\beta})\right)+(1-\beta)^2r^{-\beta}V''(r^{1-\beta}). 
\]
Thus, a sufficient condition to guarantee the concavity of $r\mapsto r^\beta V(r^{1-\beta})$ is that $V''\leq 0$ in $(0, \infty)$ and $V(0)\geq 0$. In fact, these conditions yield
\[
V'(r^{1-\beta})-{1\over r^{1-\beta}}V(r^{1-\beta})= {1\over r^{1-\beta}}\left(\int_0^{r^{1-\beta}} sV''(s)\, ds-V(0)\right)\leq 0
\]
and therefore $(r^\beta V(r^{1-\beta}))''\leq 0$ for $r>0$. In order for (F1)(F2) to hold, we also need to assume that $V$ is nondecreasing and bounded. 

In this case, we can construct $\phi$ satisfying (I) in a similar way as in \eqref{setting-phi} if  $u_0\in UC(\R^n)$ satisfies 
\begin{equation}\label{setting-u0}
 \lim_{R \to \infty} \inf_{|x| \ge R} \frac{u_0(x)}{|x|} > 0, \quad \inf_{x \in \mathbb{R}^n} u_0(x) \ge c_0 
 \end{equation}
for some $c_0>0$. We may choose an increasing function $\sigma \in C^\infty([0,\infty)) \cap \UC([0,\infty))$ such that \eqref{setting-sigma} holds and $\sigma'(s)>0$ when $s>0$ is large.
%\[ \inf_{|x| \ge R} u_0(x) \ge \sigma (R) \ge c_0 \ \text{for all $R \in [0, \infty)$}, %\quad \lim_{R \to \infty} \sigma(R), 
%\quad \inf_{R \in [0,\infty)} \sigma'(R) > 0. \]
%Since the last condition holds if $\sigma$ increases by at least the linear order, a function $\sigma$ satisfying all of the above conditions can be constructed 
Note that the second assumption in \eqref{setting-u0} is needed because of the $u$-dependent term $V(u)$ in the equation \eqref{nonlocal-evolution}. 
Using this function $\sigma$, the subsolution $\phi$ can be constructed as in \eqref{setting-phi} by replacing $\gamma$ and $C$ respectively by 
\[ \gamma(x) := |x|, \quad C:= \left(\sup_{r \in \mathbb{R}} V(r)\right)/\left(\inf_{R \in [0,\infty)} \sigma'(R)\right) + m(K). \]

\subsection{Viscous Hamilton-Jacobi equations with sublinear gradient terms}\label{sec:vishj}
Our result applies to local problems as well. For example, we can consider 
\begin{equation}\label{hj eq}
u_t-\Delta u+ a|\nabla u|^\alpha=0 %b|\nabla u|
\end{equation}
in $\R^n\times (0, \infty)$, where $a>0$ and $\alpha\in (0, 1)$ are given coefficients. 
%Note that the quasiconvexity preserving property fails to hold for the usual heat equation, i.e., \eqref{hj eq} with $a=0$ \cite{IS}. 

%Our nonlinear situation however is different. 
The operator $F$ corresponding to \eqref{hj eq} is 
\[
F(p, X)=-\tr X+a|p|^\alpha.
\]
It is clear that (F1)--(F6) hold. 
Let us consider solutions $u$ that are uniformly positive in $\R^n\times [0, \infty)$ and Lipschitz, coercive in space. Assume that $u\geq c_0$ in  $\R^n\times [0, \infty)$  for some $c_0>0$ and $u$ is $L$-Lipschitz in space. %{\color{red}(see \cite[Lemma 2.2]{BS} for the verification)}. 
By Remark \ref{rem lip}, for the assumption (F7), we only need to verify the concavity of $(r, X)\mapsto G_\beta(r, p, X)$ in $[c_0, \infty)\times \bS^n$ for all $p\in \R^n$ with $|p|\leq L$ when $0<\beta<1$ is taken sufficiently close to $1$.

Note that 
\begin{equation}\label{ex:vHJ}
G_\beta(r, p, X)=-\tr X+\beta r^{-1} |p|^2+a(1-\beta)^{\alpha-1} r^{\beta(1-\alpha)}|p|^\alpha.
\end{equation} 
It suffices to show 
\[
V(r, p)=\beta r^{-1} |p|^2+a(1-\beta)^{\alpha-1} r^{\beta(1-\alpha)}|p|^\alpha
\]
is concave with respect to $r\in [c_0, \infty)$ for $\beta$ arbitrarily close to $1$. By direct calculations, we get
\[
\begin{aligned}
V_{rr}(r, p)&= 2\beta r^{-3}|p|^2+a(1-\beta)^{\alpha-1} \beta(1-\alpha)(\beta-\alpha\beta-1)r^{\beta(1-\alpha)-2}|p|^\alpha\\
&=r^{-3}|p|^\alpha \left(2\beta|p|^{2-\alpha}+a(1-\beta)^{\alpha-1} (\beta-\alpha\beta) (\beta-\alpha\beta-1)r^{\beta(1-\alpha)+1}\right).
\end{aligned}
\]
For $|p|\leq L$ and $r\geq c_0$, we have 
\[
\begin{aligned}
&2\beta|p|^{2-\alpha}+a(1-\beta)^{\alpha-1} (\beta-\alpha\beta) (\beta-\alpha\beta-1)r^{\beta(1-\alpha)+1}\\
&\leq 2L^{2-\alpha}+a(1-\beta)^{\alpha-1}(\alpha-\alpha^2)(\alpha-\alpha^2-1)\min\{c_0^{\alpha(1-\alpha)+1},  c_0^{2-\alpha}\}
\end{aligned}
\]
for all $\alpha\leq \beta <1$. Noticing that $\alpha-\alpha^2-1<0$ and $(1-\beta)^{\alpha-1}\to \infty$ as $\beta\to 1$, when $0<\beta<1$ is sufficiently close to $1$, we have 
\[
2\beta|p|^{2-\alpha}+a(1-\beta)^{\alpha-1} \beta(1-\alpha) (\beta-\alpha\beta-1)r^{\beta(1-\alpha)+1}<0
\]
and therefore $V_{rr}(r, p)\leq 0$ for $|p|\leq L$ and $r\geq c_0$.

In this case, we still have the existence of the function $\phi$ satisfying (I) if $u_0\in UC(\R^n)$ fulfills \eqref{setting-u0}. It can be constructed as 
\[ \phi(x,t) := % \begin{cases}
k\left(\max\left\{ |x| - R -ak^{\alpha-1} t,\ 0\right\}\right) + c_0 %&\text{for} \; \; |x| \ge R + ad^{\alpha-1} t, \\
%c_0 & \text{for} \; \; |x| < R+ad^{\alpha-1}t
%\end{cases}
\] 
for some $k, R > 0$. 
 %since this local equation \eqref{hj eq} is far away from geometric, we %cannot apply %the geometrical property 
%\eqref{geometrical-pro} and
 %we need to construct $\phi$ in a way much different from \eqref{}. to 
% choose $\sigma$ to control the term $\Delta u$ if we construct $\phi$ of a form \eqref{setting-phi} with $\gamma(x) = |x|$. 
By direct calculations, we get, for $|x|> R+ak^{\alpha-1} t$, 
\[
\phi_t(x, t)-\Delta \phi(x, t)+a|\nabla \phi(x, t)|^\alpha= -ak^\alpha-{n-1\over |x|}+ak^\alpha\leq 0.
\]
It is thus not difficult to verify that such a function $\phi$ meets our needs. 

Finally it is worth emphasizing that the quasiconvexity preserving property fails to hold for the usual heat equation, i.e., \eqref{hj eq} with $a=0$, as pointed out in \cite{IS, IST2}. 
Indeed, in our analysis, the function $G_\beta$ given by \eqref{ex:vHJ} with $a=0$ does not satisfy (F7).

\bibliographystyle{abbrv}

%\end{comment}

\end{document}